\documentclass[12pt]{article}
\usepackage{fullpage}
\usepackage{verbatim}
\usepackage{amsthm}
\usepackage{caption}
\usepackage{amsmath,amssymb}
\usepackage{xcolor,tikz,tkz-graph}
\usepackage{graphicx}

\usepackage[colorlinks=false,linkbordercolor=red]{hyperref}
\hypersetup{%
  colorlinks=false,
  linkbordercolor=red,
  pdfborderstyle={/S/U/W 1}
}

\newtheorem*{named1}{Theorem~1}
\newtheorem*{named2}{Theorem~2}
\newtheorem{lemma}{Lemma}
\newtheorem{thm}[lemma]{Theorem}
\newtheorem{cor}[lemma]{Corollary}
\newtheorem{claim}{Claim}
\newtheorem{observation}{Observation}

\theoremstyle{definition}

\newtheorem{conj}{Conjecture}
\newtheorem{defn}{Definition}
\def\card#1{\left|#1\right|}

\def\A{\mathcal{A}}
\def\B{\mathcal{B}}
\def\C{\mathcal{C}}
\def\E{\mathcal{E}}
\def\G{\mathcal{G}}
\def\H{\mathcal{H}(G)}
\def\S{\mathcal{S}}
\def\adj{\leftrightarrow}
\def\nonadj{\not\leftrightarrow}
\def\floor#1{\lfloor #1\rfloor}
\def\Floor#1{\left\lfloor #1\right\rfloor}
\def\ceil#1{\lceil #1\rceil}

\newcommand{\set}[1]{\left\{ #1 \right\}}
\newcommand{\irange}[1]{\left[#1\right]}
\newcommand{\join}[2]{#1 \mbox{\hspace{2 pt}$\vee$\hspace{2 pt}} #2}

\newcommand{\IN}{\mathbb{N}}
\newcommand{\parens}[1]{\left( #1 \right)}

\newcommand{\setbs}[2]{\left\{ #1 \mid #2 \right\}}
\newcommand{\size}[1]{\left\Vert#1\right\Vert}

\newcommand{\brackets}[1]{\left[ #1 \right]}
\renewcommand{\sp}{\operatorname{sp}}
\newcommand{\esp}{\operatorname{esp}}

\begin{document}
\title{
Graphs with $\chi=\Delta$ have big cliques
}
\author{Daniel W. Cranston\thanks{Department of Mathematics and Applied
Mathematics, Virginia Commonwealth University, Richmond, VA, 23284. email:
\texttt{dcranston@vcu.edu}} \and 
Landon Rabern\thanks{Lancaster, PA, 17601.  email:
\texttt{landon.rabern@gmail.com}. 
}}
\maketitle
\date


\abstract{
Brooks' Theorem implies that if a graph has $\Delta\ge 3$ and
and $\chi > \Delta$, then $\omega=\Delta+1$.
Borodin and Kostochka conjectured that if $\Delta\ge 9$ and $\chi\ge \Delta$,
then $\omega\ge \Delta$.  We show that if $\Delta\ge 13$ and $\chi\ge \Delta$,
then $\omega \ge \Delta-3$.
For a graph $G$, let $\H$ denote the subgraph of $G$ induced by vertices of
degree $\Delta$.  We also show that if $\chi\ge \Delta$, then $\omega\ge
\Delta$ or $\omega(\H)\ge \Delta-5$. 
}

\section{Introduction}
Our goal in this paper is to prove the following two main results.
For a graph $G$, we write $\Delta(G)$, $\omega(G)$, and $\chi(G)$ to denote
(respectively) the maximum degree, clique number, and chromatic number of $G$. 
When the context is clear, we simply write $\Delta$, $\omega$, and $\chi$.

\begin{thm}
\label{main1}
If $G$ is a graph with $\chi\ge\Delta\ge 13$, then $\omega\ge \Delta-3$.
\end{thm}

\begin{thm}
\label{main2}
Let $G$ be a graph and let $\H$ denote the subgraph of $G$ induced by
vertices of degree $\Delta$.  If $\chi\ge \Delta$, then $\omega\ge \Delta$ or
$\omega(\H)\ge \Delta-5$.
\end{thm}

The proofs of Theorems~\ref{main1} and \ref{main2} are
both somewhat detailed, so we first prove Theorem~\ref{firstthm}, which 
plays a central role in proving our two main theorems.
(For a less formal and less notationally dense presentation of these results,
see an earlier version of this paper~\cite{big-cliques-arxiv1}.)
Brooks' Theorem states that if $G$ is connected and $\chi > \Delta$, then
$G$ is a complete graph on $\Delta+1$ vertices (in particular,
$\omega=\Delta+1$) or $G$ is an odd cycle;
so if $\Delta\ge 3$, then
$\chi>\Delta$ implies $\omega=\Delta+1$.  Thus, the interesting case of
Theorems~\ref{main1} and \ref{main2} is when $\chi=\Delta$.

\begin{thm}
\label{firstthm}
If $G$ is a graph with $\chi \ge \Delta$, then $\omega \ge \Delta - 3$ if $\Delta \equiv 1\pmod 3$ and $\omega \ge \Delta - 4$ otherwise.
\end{thm}

When $\Delta=13$, Theorem~\ref{firstthm} implies that either $G$ is
12-colorable or $G$ contains a $K_{10}$.  This result will serve as the base
case for a proof of Theorem~\ref{main1} by induction on $\Delta$.  To prove
Theorem~\ref{main2}, we will further analyze the proof of
Theorem~\ref{firstthm}, and show that we can continue a certain recoloring
process unless $\H$ contains a big clique.

Borodin and Kostochka~\cite{BK} conjectured in 1977 that if $G$ is a graph with
$\Delta\ge 9$ and $\omega\le\Delta-1$, then $\chi\le\Delta-1$.  The hypothesis
$\Delta\ge9$ is needed, as witnessed by the following example.  Form $G$ from
five disjoint copies of $K_3$, say $D_1,\ldots,D_5$, by adding edges between $u$
and $v$ if $u\in D_i$, $v\in D_j$, and $i-j\equiv1\bmod 5$.  This graph is
8-regular with $\omega=6$ and $\chi\ge\ceil{15/2}=8$, since each color is used
on at most 2 of the 15 vertices; by Brooks' Theorem $G$ is 8-colorable, so
$\chi(G)=8$.  Various other examples with $\chi=\Delta$ and
$\omega<\Delta$ are known for $\Delta\le 8$ (see for example~\cite{CR2}).
The Borodin-Kostochka Conjecture has been proved for various families of graphs.
Reed~\cite{Reed} used probabilistic arguments to prove it for graphs with
$\Delta\ge 10^{14}$.  The present authors~\cite{CR2} proved it for claw-free
graphs (those with no induced $K_{1,3}$).

The contrapositive of the conjecture states that if $\chi\ge \Delta\ge 9$, then $\omega\ge\Delta$.  
The first result in this direction was
due to Borodin and Kostochka~\cite{BK}, who proved that
$\omega\ge\floor{\frac{\Delta+1}{2}}$ when $\chi \ge \Delta$.  Subsequently, Mozhan~\cite{Moz1} improved this to
$\omega\ge\floor{\frac{2\Delta+1}{3}}$ when $\Delta \ge 10$ and Kostochka~\cite{Kostochka} showed that $\chi \ge \Delta$ implies
that $\omega\ge\Delta-28$.  
Finally, Mozhan proved that $\omega\ge\Delta-3$ when $\chi\ge \Delta\ge 31$
(this result was in his Ph.D. thesis, which unfortunately is
not readily accessible~\cite{Reed}).  Theorem~\ref{main1} strengthens Mozhan's
result, by weakening the condition to $\Delta\ge 13$.  
Work in the direction of Theorem~\ref{main2} began in~\cite{kierstead2009ore},
where Kierstead and Kostochka proved that if $\chi\ge\Delta\ge 7$ and
$\omega\le\Delta-1$, then $\omega(\H)\ge 2$.  This was strengthened
in~\cite{KRS} to the conclusion $\omega(\H)\ge\floor{\frac{\Delta-1}2}$.  We
further strengthen the conclusion to $\omega(\H)\ge \Delta-5$. 
We give more background
in the introduction to Section~\ref{sectionMain1}.

Most of our notation is standard, as in~\cite{IGT}.
We write $K_t$ and $E_t$ to denote the complete and edgeless graphs on $t$
vertices, respectively.  A subset of vertices $S$ is a \emph{clique} if
$S$ induces a complete graph.  We write $[n]$ to denote $\{1,\ldots,n\}$. 
The \textit{join} of disjoint graphs $G$ and $H$,
denoted $\join{G}{H}$, is formed from the disjoint union of $G$ and $H$ by
adding all edges with one endpoint in each of $G$ and $H$.  
Two sets of vertices $R$ and $S$ in a graph $G$ are \emph{joined} if for every
pair of distinct vertices $r, s$ with $r\in R$ and $s\in S$, the graph $G$
contains the edge $rs$. (Note that $R$ and $S$ need not be disjoint.) Subgraphs $A$ and $B$ of $G$ are \emph{joined} if $V(A)$ and $V(B)$ are joined.
If $R$ and $S$ are joined to each other, we may also say that $R$ is
\emph{complete} to $S$.

For a vertex $v$ and a set $S$ (containing
$v$ or not) we write $d_S(v)$ to denote $|S\cap N(v)|$.  When vertices $x$ and
$y$ are adjacent, we write $x\adj y$; otherwise $x\nonadj y$.  If $\mathcal{Z}$ is a set of graphs, we let $V(\mathcal{Z}) = \bigcup_{G \in \mathcal{Z}} V(G)$.
A graph $G$ is \emph{$k$-critical} if $\chi(G)=k$ and $\chi(H)<k$ for every
proper induced subgraph $H$.  (When we say simply that a graph $G$ is
\emph{critical}, we mean that is $\chi(G)$-critical.)  A vertex $v$ in a graph
$G$ is \emph{critical} if $\chi(G\setminus\{v\})<\chi(G)$.  Note that in a
$\Delta$-critical graph, every vertex has degree $\Delta$ or $\Delta-1$.  
A vertex $v$ is \textit{high} if $d(v)=\Delta$ and \textit{low} otherwise.  

\section{Mozhan's Partitioned Colorings}
\label{mozhan-partitions}

In \cite{Moz1}, Mozhan used a partition of a graph into groups of color classes
to prove bounds on the chromatic number in terms of the degree and clique
number.  These ideas trace all the way back to the 1966 paper of Lov{\'a}sz
\cite{lovasz1966decomposition} where he proves that if $G$ is a graph and $r_1,
\ldots, r_k \in \IN$ with $\sum_{i \in \irange{k}} r_i \ge \Delta(G) + 1 - k$,
then $V(G)$ has a partition $\set{V_1, \ldots, V_k}$ where $\Delta(G[V_i]) \le
r_i$ for all $i \in \irange{k}$.  The proof idea is simple; just take a
partition minimizing the number of edges within parts (with an appropriate
weighting depending on $r_i$).  In \cite{CatlinAnotherBound}, Catlin took this
idea further by starting with such a minimum partition and then moving
vertices
around (while preserving minimality) until he achieved a desired property.  To
get the ability to move vertices around like this, he needed to strengthen the
condition on the $r_i$ to $\sum_{i \in \irange{k}} r_i \ge \Delta(G) + 2 - k$. 

Mozhan's idea is very similar to Catlin's, but not equivalent.  As we will see
below, Mozhan considers partitions of $V(G)$ minimizing the number of edges
within parts, just
like Lov{\'a}sz and Catlin, but he adds the restriction that each part is the
disjoint union of color classes in some fixed $\chi(G)$-coloring of $G$.  With
this added restriction we get a weaker bound on the degrees within parts, but
more information about the coloring.  Because of this trade-off Mozhan's method
excels when all we care about is coloring the parts, but if we require the
parts to have more structure (for example, for them to be degenerate as in
Borodin's result \cite{borodin1976decomposition}), we need to use Catlin's
method or some other technique (see \cite{borodin2000variable} for example). 
In some cases either technique will work; Mozhan's method was used
in~\cite{Rabern2} and~\cite{KRS}, but the same results were derived in
\cite{Rabern3} using Catlin's method.  The results in this paper require the
use of Mozhan's more restrictive partitions, which we define now.

In our proofs of Theorems~\ref{main1}, \ref{main2}, and \ref{firstthm}, we assume that
$G$ is critical, so we only need the partition in the following definition when $G$ is
critical.  However, we include non-critical graphs as well because the more general
concept is needed to extract an efficient algorithm from our proofs.  We discuss
algorithmic considerations in the final section of the paper.  
Since the proof of Theorem~\ref{firstthm} is long, we provide a proof sketch as
soon as we have the necessary definitions.  This immediately follows
Definition~\ref{full}.

\begin{defn}
\label{Mpartition}
For $s \in \IN_{\ge 2}$ and $r_1, \ldots, r_s \in \IN_{\ge 3}$, an \emph{$(r_1,
\ldots, r_s)$-partition} $P$ of a graph $G$ is a partition $(P_1,\ldots,P_s)$ of
$V(G)$, together with an integer $j\in [s]$ and vertex $v\in P_j$, such that
\begin{enumerate}
\item[(1)] 
$\chi(G[P_i])=r_i$ for all $i\in [s]\setminus\{j\}$; and
\item[(2)] 
$\chi(G[P_j] \setminus \{v\}) \le r_j$.  
\end{enumerate}

\noindent We refer to $j$ and $v$ by $j(P)$ and $v(P)$ respectively.
\end{defn}

For example, if $G$ is critical and $\Delta(G)=13$, then we get a
$(3,3,3,3)$-partition of
$G$ by removing any $v \in V(G)$, partitioning the color classes of a
$12$-coloring of $G-v$ into four equal parts and then adding $v$ to one part,
called part $j$. 

We are interested in $(r_1, \ldots, r_s)$-partitions that minimize the total
number of edges within parts (without $v(P)$).  
More precisely, for an $(r_1, \ldots, r_s)$-partition $P$ of a graph $G$, let
$\sigma(P) = \size{G[P_{j(P)}] \setminus \{v(P)\}} + \sum_{i\in
[s]\setminus\{j(P)\}} \size{G[P_i]}$; here $\size{H}$ denotes the number of
edges in subgraph $H$. A \emph{minimum $(r_1, \ldots, r_s)$-partition} of $G$
is an $(r_1, \ldots, r_s)$-partition $P$ minimizing $\sigma(P)$ and, subject to
that, minimizing $d_{j(P)}(v(P))-r_{j(P)}$.

\begin{lemma}
\label{GeneralMozhan}
If $P$ is a minimum $(r_1, \ldots, r_s)$-partition of a graph $G$ with $\chi(G)
= \Delta(G) = 1 + \sum_{i \in \irange{s}} r_i$, then
\begin{enumerate}
\item[(1)] $G[P_{j(P)}]$ has a component $\A(P)$, called the \textit{active}
component, that is $K_{r_{j(P)} + 1}$ and $\chi(G[P_{j(P)}]\setminus
V(\A(P)))\le r_{j(P)}$; and
\item[(2)] for each $u \in V(\A(P))$ and $i \in[s]\setminus \{j(P)\}$ with
$d_{P_i}(u) = r_i$, the graph $G[P_i \cup \{u\}]$ has a $K_{r_i + 1}$
component (which contains $u$); and
\item[(3)] for each $u\in V(\A(P))$ and $i \in [s] \setminus\{j(P)\}$, if $u$
has at least $d_{P_i}(u) + 1 - r_i$ neighbors in the same component $D$ of
$G[P_i]$, then $\chi(G[V(D) \cup\{u\}]) = r_i + 1$; and
\item[(4)] if $u \in V(G)$ and $a \in \irange{s}$ so that $d_{P_a}(u) > r_a +
1$, then there is $i \in \irange{s}$ where $d_{P_i}(u) < r_i$.  
In particular, any $r_i$-coloring of $G[P_i]$ can be extended to an $r_i$
coloring of $G[P_i \cup \set{u}]$; and
\item[(5)] for each $u\in V(\A(P))$ and $i \in [s] \setminus\{j(P)\}$, we have
$d_{P_i}(u) \le r_i + 1$.
\end{enumerate}
\end{lemma}
\begin{proof}
Let $P$ be a minimum $(r_1, \ldots, r_s)$-partition of a graph $G$ with
$\chi(G) = \Delta(G) = 1 + \sum_{i \in \irange{s}} r_i$.
Let $j = j(P)$ and $v = v(P)$.
Let $\A(P)$ be the component of $G[P_j]$ containing $v$.  
Fix a $(\Delta(G)-1)$-coloring of $G-v$ consistent with the partition.
Since
$\sum_{i\in[k]}r_i=\Delta(G)-1$, and since $d_{j(P)}(v)-r_{j}$ is minimized
in the choice of the partition, we must have $d_{j}(v)\le r_j$.  Equality must
hold, since otherwise we could extend the $(\Delta(G)-1)$-coloring of $G-v$ to $v$.

By construction, $G[P_j\setminus\{v\}]$ has an $r_j$-coloring.
So we may assume that $\chi(\A(P))=r_j+1$, since otherwise we get an
$r_j$-coloring of $G[P_j]$, and hence a $(\Delta-1)$-coloring of $G$.

To prove (1), it suffices to show that $\A(P)$ is $K_{r_j + 1}$.  By Brooks'
Theorem, it is enough to show that $\Delta(\A(P)) \le r_j$.  Suppose instead
that there exists $u\in V(\A(P))$ with $d_{\A(P)}(u) > r_j$; choose $u$ to
minimize the distance in $\A(P)$ from $u$ to $v$.  Uncolor the vertices on a
shortest path $Q$ in $\A(P)$ from $u$ to $v$; move $u$ to some $P_k$ where it
has at most $r_k$ neighbors.  Color the vertices of $Q$,
starting at $v$ and working along $Q$; this is possible since each vertex of
$Q$ has at most $r_j - 1$ colored neighbors in $\A(P)$ when we color it.  
The resulting new partition $R$ (with $v(R)=u$) has fewer edges within color
classes, since we lost at least $r_j + 1$ edges incident to $u$ and gained at
most $r_j$ incident to $v$. This contradiction implies that $\Delta(\A(P))\le
r_j$, so $\A(P)$ must be $K_{r_j + 1}$ by Brooks' Theorem.  Thus (1) holds.  

Now we prove (2).  Choose such a vertex $u\in V(\A(P))$ and such an
$i\in[s]\setminus\{j\}$.  Form a new partition $R$ by deleting $u$ from $P_j$
and adding it to $P_i$ (now $u=v(R)$); this maintains the total number of edges
within parts, so $R$ is another minimum $(r_1, \ldots, r_s)$-partition.  
By the above proof of (1), $u$ lies in a component of $G[P_i]$
that is $K_{r_j + 1}$.  Thus, (2) holds.  

If (3) is false, then $u$ has at most $r_i - 1$
neighbors in $G[P_i] \setminus D$, so we may choose an $r_i$-coloring of
$G[P_i] \setminus D$ so that the neighbors of $u$ in $P_i \setminus V(D)$ each
get a color in $\irange{r_i - 1}$.  
Together with an $r_i$-coloring of $G[V(D)\cup\{u\}]$ where $u$ is colored
$r_i$, this gives an $r_i$-coloring of $G[V(P_i) \cup\{u\}]$.  But then we have a
$(\chi(G)-1)$-coloring of $G$, a contradiction.

(4) is immediate,  since $d_G(u) \le 1 + \sum_{i \in \irange{s}} r_i$ 

If (5) is false, then apply (4) and move $u$ to $P_i$ to get a $(\chi(G)-1)$-coloring of $G$, a contradiction.
\end{proof}

\begin{defn}
A \emph{move} is a quadruple $(P, u, i, P')$ where 
\begin{enumerate}
\item[(1)] $P$ is an $(r_1, \ldots, r_s)$-partition of a graph $G$; and
\item[(2)] $u \in V(\A(P))$; and
\item[(3)] $i \in[s] \setminus \{j(P)\}$ with $d_{P_i}(u) = r_i$; and 
\item[(4)] $P'$ is obtained from $P$ by moving $u$ from $P_{j(P)}$ to $P_i$.
\end{enumerate}
In $P'$, vertex $v(P)$ is in the part containing $V(\A(P)\setminus\{u\})$.  Also
$j(P')=i$ and $v(P')=u$.
\end{defn}

In the proof of part (2) of Lemma \ref{GeneralMozhan}, we showed that if $P$ is a minimum $(r_1, \ldots, r_s)$-partition 
and $(P, v, i, P')$ is a move, then $P'$ is a minimum $(r_1, \ldots, r_s)$-partition as well.  
Moreover, for each $k \in \irange{s}$, the number of components in $G[P_k]$ equals the number of components in $G[P'_k]$.

\begin{defn}
Let $P$ be an $(r_1, \ldots, r_s)$-partition of a graph $G$.  A \emph{move
sequence} starting at $P$ is a sequence of moves $((P^1, v_1, i_1, P^2),
\ldots, (P^q, v_q, i_q, P^{q+1}))$ where ${P^1 = P}$.
\end{defn}

\begin{defn}
Let $P$ be an $(r_1, \ldots, r_s)$-partition of a graph $G$ and 
\[\S = ((P^1,v_1, i_1, P^2),\ldots, (P^q, v_q, i_q, P^{q+1}))\] 
\noindent a move sequence starting at
$P$. For each $i \in \irange{s}$ and component $X$ of $G[P_i]$, let the
$\emph{club}$ of $X$, written $\C_{\S}(X)$, be the sequence $(X_1, X_2, X_3,
\ldots, X_{q+1})$ where $X_1 = X$ and for $t \in \irange{q} \setminus \set{1}$
\begin{itemize}
\item $X_t = X_{t-1} \setminus \{v_{t-1}\}$ if $X_{t-1}$ is the active
component in $P^{t-1}$; otherwise 
\item $X_t = X_{t-1} \cup \{v_{t-1}\}$ if $G[V(X_{t-1}) \cup \{v_{t-1}\}]$ is
the active component in $P^t$; otherwise
\item $X_t = X_{t-1}$.
\end{itemize}
\end{defn}

We need to extend the domain of $\C_\S$ to all components at all times in a
given sequence. To do this consistently, we will let $\C_\S(Y)$ be the club
that $Y$ appears in most recently.  Now we give a precise definition of this extension.
Let $P$ be an $(r_1, \ldots, r_s)$-partition of a graph $G$ and 
\[\S = ((P^1,v_1, i_1, P^2),\ldots, (P^q, v_q, i_q, P^{q+1}))\] 
\noindent a move sequence starting at $P$.  For each $t \in \irange{q+1}$ and $Y$ a
component of $G[P_i^t]$ for some $i \in \irange{s}$, we define $\C^t_\S(Y)$ to be
$\C_{\S}(X) = (X_1, X_2, X_3,\ldots, X_{q+1})$ where $X$ is the component of
$G[P_i^1]$ such that $V(Y) = V(X_t)$.  Often, the time $t$ will be
clear from context, so we can write simply $\C_\S(X)$.

When the move sequence is clear from context, we write $\C(X)$ in place of
$\C_{\S}(X)$. We say \emph{$R$ is a club of $\S$} if $R = \C_\S(X)$ for a component $X$ of $G[P_i]$ for some $i \in \irange{s}$. For a club $R$, we write $R_t$ for the $t$-th element of $R$.  

We observe a few basic facts about clubs; we omit formal proofs by induction,
which are easy exercises.

\begin{observation}\label{BasicClubFacts1}
Let $P$ be a minimum $(r_1, \ldots, r_s)$-partition of a graph $G$ with $\chi(G) = \Delta(G) = 1 + \sum_{i \in \irange{s}} r_i$.  If
\[\S = ((P^1,v_1, i_1, P^2),\ldots, (P^q, v_q, i_q, P^{q+1}))\] is a move sequence starting at $P$, then for a club $R$ of $\S$, we have
\begin{enumerate}
\item[(1)] if $V(R_1) \subseteq P^1_i$, then $V(R_t) \subseteq P^t_i$ for all $t \in \irange{q+1}$.  We call this $i$ the \emph{part} of $R$, written $\rho_\S(R)$ (or $\rho(R)$ when context allows).
\item[(2)] if $a,b \in \irange{q+1}$, then $R_a$ is complete if and only if $R_b$ is complete.
\item[(3)] if $R_t$ is complete and $\card{R_t} \ge r_{\rho(R)}$ for all $t \in
\irange{q+1}$,
then $\card{R_a} = r_{\rho(R)} + 1$ when $R_a$ is active and otherwise $\card{R_a} = r_{\rho(R)}$.
\end{enumerate}
\end{observation}
The notion introduced in (3) of the previous observation is important, so, in the
following definition, we name it.

\begin{defn}
\label{full}
Let $P$ be a minimum $(r_1, \ldots, r_s)$-partition of a graph $G$ with $\chi(G) = \Delta(G) = 1 + \sum_{i \in \irange{s}} r_i$.  Let
\[\S = ((P^1,v_1, i_1, P^2),\ldots, (P^q, v_q, i_q, P^{q+1}))\] 
be a move sequence starting at $P$.  
A club $R$ of $\S$ is \emph{full} if $R_t$ is complete and $\card{R_t} \ge r_{\rho(R)}$ for all $t \in \irange{q+1}$.
\end{defn}

At this point we have enough definitions to outline the plan for proving
Theorem~\ref{firstthm}.  We start with a Mozhan partition (as in
Definition~\ref{Mpartition}) and repeatedly move a vertex from the active
component; our goal is either to find a $(\Delta-1)$-coloring or a copy of
$K_{\Delta-4}$ (before we reach a club with no unmoved vertices).  Our move
sequence will
satisfy the following criteria: each vertex moves at most once; a vertex never
moves from a club $R$ to a club $S$ if $R$ and $S$ are joined; if possible 
the active club sends a vertex to a club to which it has already sent a vertex.

Now each vertex in the active component can be sent to any of all but at most
one other clubhouse due to degree considerations. If it cannot be sent to some
additional clubhouse, this is because the active component, say $R$, is joined
to a full club, say $S$, in that clubhouse (Definition~\ref{JoinedDefn} defines
two full clubs, $R$ and $S$, being joined, but it implies that $V(R)$ is joined to $V(S)$,
which is enough for now).  The main idea is that when a maximal such move sequence
stops, it is because the active component is joined to full clubs in all but at most
one of the other clubhouses.  The final ingredient is to show that full clubs being
joined is a transitive relation; that is, if clubs $R$, $S$, and $T$ are full
and $R$ is joined to $S$ and $T$, then also $S$ is joined to $T$.  This implies that
at the end of a maximal move sequence all of the full clubs joined to the active
component are joined to each other, and thus induce a big clique (in fact, size
at least $\Delta-4$).

\begin{lemma}\label{HelperPrecursor}
Let $H$ be a graph with induced subgraphs $A_1, \ldots, A_k$ where
$\{V(A_1), \ldots, V(A_k)\}$ partitions $V(H)$ and $\chi(H) = \sum_{i \in
\irange{k}} \chi(A_i)$ where $\chi(A_1) \ge 4$ and $\chi(A_i) \ge 3$ for all $i
\in \irange{k} \setminus \set{1}$.
Let $u\in V(A_1)$ be such that $\chi(A_1\setminus\{u\})< \chi(A_1)$ and let
$T_1$ be the component of $A_1$ containing $u$.  Now $\chi(T_1)=\chi(A_1)$ and
the next three statements hold.

\begin{enumerate}
\item[(a)] For each $i\in [k]\setminus \{1\}$ there is a component $T_i$ of $A_i$
such that $\chi(T_i)=\chi(A_i)$ and $d_{V(T_i)}(u)\ge \chi(A_i)$.

\item[(b)]  
Define $T_i$, for all $i\in [k]$, as above.
Suppose $d_{V(T_k)}(u) = \chi(A_k)$ and $d_{V(A_k)}(u) \le \chi(A_k) + 1$.  
Put $A^* = V(\set{A_1, \ldots, A_{k-1}})$ and $T^* = V(\set{T_1, \ldots,
T_{k-1}})$.  Further suppose there is $v \in N(u) \cap V(T_k)$ with $d_{A^*}(v)
\le 1 + \sum_{i \in \irange{k-1}} \chi(A_i)$ and $d_{T^*}(v) \ge 3$. Now there
exists $q \in \irange{k-1}$ such that $d_{V(T_q)}(v) \ge \chi(A_q)$. 

\item[(c)] Define $A^*$, $T^*$, and $v$ as in (b).  If $T^*$ induces a clique,
$T_k$ is complete, and $d_{A^*}(w) \le \card{T^*}$ for all $w \in T^*$, then
$T^* \cup \set{v}$ induces a clique.
\end{enumerate}
\end{lemma}

\begin{proof}
First we prove (a).
Pick $i \in \irange{k} \setminus \set{1}$. Since $\chi(A_1 \setminus\{u\}) <
\chi(A_1)$, we must have $\chi(A_i') = \chi(A_i) + 1$, where $A_i' = G[V(A_i)
\cup \set{u}]$.  So, $u$ has at least $\chi(A_i)$ neighbors in some component
$T_i$ of $A_i$, for otherwise we get a $\chi(A_i)$-coloring of $A_i'$ from a
$\chi(A_i)$-coloring of $A_i$ by permuting colors in components of $A_i$. This
proves (a).

Now we prove (b).  Our plan is to move $u$ to part $A_k$ and move $v$ to some
other part, and show that if (b) fails, then we have a $(\chi(G)-1)$-coloring.
Put $A_1' = G\brackets{V(A_1 \setminus\{u\}) \cup \set{v}}$ and 
$A_k' = G\brackets{V(A_k) \cup\set{u} }$ and 
$A_i' = G\brackets{V(A_i) \cup \set{v}}$ for each $i \in \irange{k-1} \setminus \set{1}$. 
Since $\chi(A_1 \setminus\{u\}) < \chi(A_1)$, we must have $\chi(A_k') =
\chi(A_k) + 1$ and $u$ is critical in $A_k'$.  Also $v$ is critical in
$A_k'$ since we can $\chi(A_k)$-color $G[A_k\setminus\set{v}]$, and extend the
coloring to $u$, since $d_{V(T_k)}(u)=\chi(A_k)$ but $v$ is removed from $T_k$
(if $u$ has one other neighbor in $A_k$, then we may possibly have to permute
the colors on that component of $G[A_k]$ to avoid the color used on $u$ in
$G[V(T_k\setminus\set{v})\cup\set{u}]$ ).
Since $v$ is critical in $A'_k$, we conclude that $d_{V(A_i)}(v) \ge \chi(A_i)$
for each $i \in \irange{k-1}$.

Since $\chi(A_k' \setminus\{v\}) < \chi(A_k')$, we must have $\chi(A_1') \ge
\chi(A_1)$ and $\chi(A_i') \ge \chi(A_i) + 1$ for each $i \in \irange{k-1} \setminus
\set{1}$.  In particular, $v$ is critical in $A_i'$ for each $i \in
\irange{k-1}$. Note that $d_{V(A_i)}(v) \le \chi(A_i) + 1$ for each $i \in
\irange{k-1}$ since $d_{A^*}(v) \le 1 + \sum_{i \in \irange{k-1}} \chi(A_i)$. 
Moreover, there is at most one $i \in \irange{k-1}$ for which $d_{V(A_i)}(v) =
\chi(A_i) + 1$.  Now the remainder of (b) consists of the following claim.

\begin{figure}
\centering
\resizebox{0.85\textwidth}{!}{\includegraphics{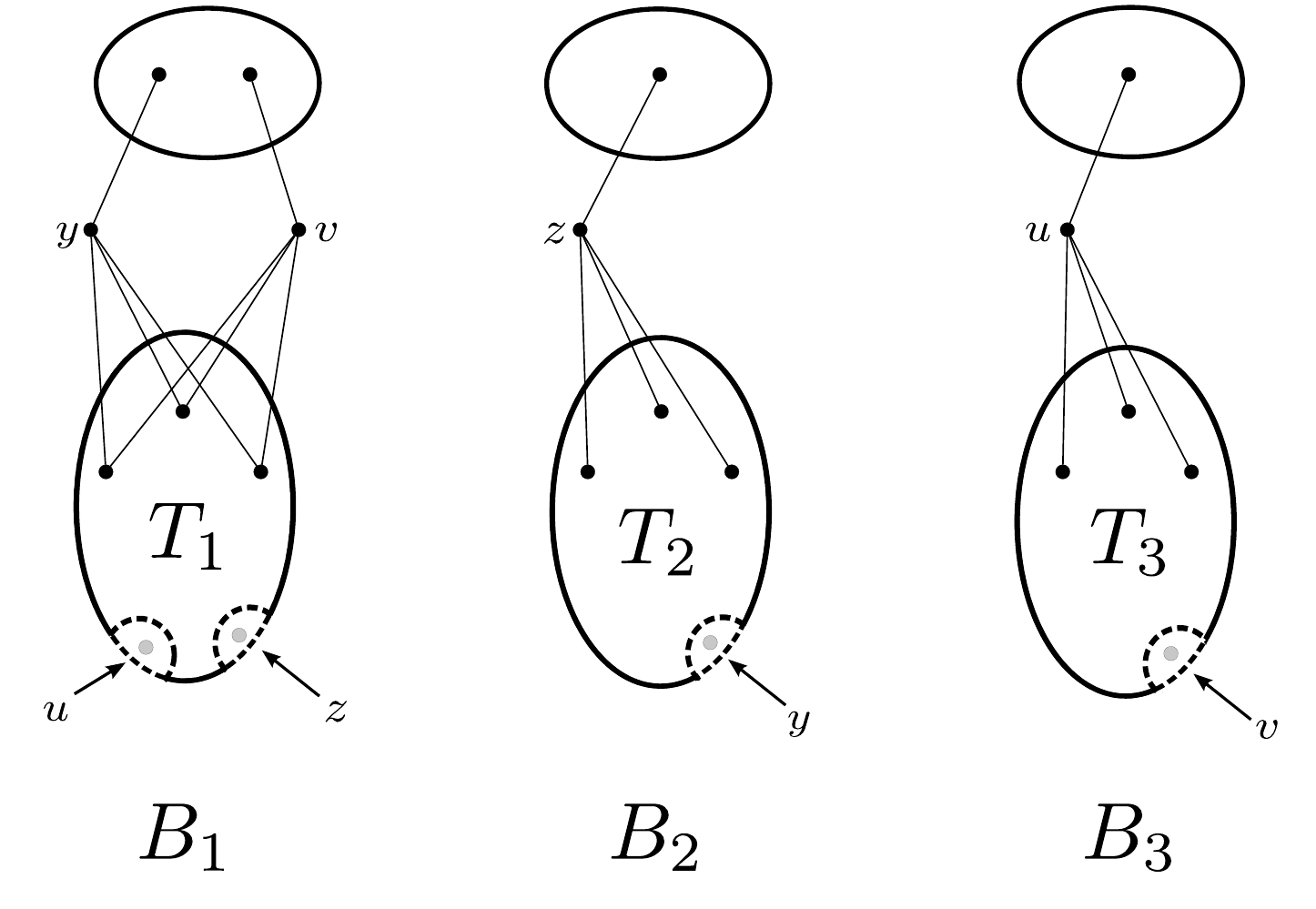}}
\captionsetup{width=0.8\textwidth}
\caption{The partition in Claim 2 of Lemma \ref{HelperPrecursor}. To form
$B_1$, $B_2$, and $B_3$ from $A_1$, $A_2$, and $A_3$ (respectively), the
vertices circled with dotted lines (and shown in gray) have now been moved to
other parts, where they are shown above the $T_i$'s.}
\label{HelperPrecursorFigure}
\end{figure}

\begin{claim}
There exists $q \in \irange{k-1}$ such that $d_{V(T_q)}(v) \ge \chi(A_q)$.
\end{claim}
Pick $w,x \in N(v) \cap T^* \setminus \set{u}$.  First, suppose there is $i \in
\irange{k-1}$ with $w,x \in V(T_i)$.  Since $v$ is critical in $A_i'$, it has
at least $\chi(A_i') - 1$ neighbors in some component $C$ of $A_i'
\setminus\{v\}$. Since $v$ has two neighbors in $T_i$, our bounds on
$d_{V(A_i)}(v)$ and $\chi(A_i')$ imply that $C = T_i$.  Since $\chi(A_i') \ge
\chi(A_i) + 1$ for each $i \in \irange{k-1} \setminus \set{1}$ (and if $i=1$,  $v$
gets $u$ as an extra neighbor), the claim is satisfied.

So, we may assume there are different $i,j \in \irange{k-1}$ with $w \in
V(T_i)$ and $x \in V(T_j)$.  Since there is at most one $p \in \irange{k-1}$
for which $d_{V(A_p)}(v) = \chi(A_p) + 1$, by symmetry we may assume that
$d_{V(A_j)}(v) = \chi(A_j)$.  Since $v$ is critical in $A_j'$, it has at least
$\chi(A_j') - 1$ neighbors in some component $C$ of $A_j' \setminus\{v\}$. 
Since $v$ has at least one neighbor in $T_j$, our bounds on $d_{V(A_j)}(v)$ and
$\chi(A_j')$ imply that $C = T_j$.  This proves the claim, and completes the
proof of (b).

Now we prove (c), which we restate as the following claim.

\begin{claim}
If $T^*$ induces a clique, $T_k$ is complete, and $d_{A^*}(w) \le \card{T^*}$ for all $w \in T^*$, then $T^* \cup \set{v}$ induces a clique.
\end{claim}
Suppose otherwise that $T^*$ induces a clique, $T_k$ is complete, and
$d_{A^*}(w) \le \card{T^*}$ for all $w \in T^*$ but $T^* \cup \set{v}$ does not
induce a clique.  By (b) we have $q \in \irange{k-1}$ such that $d_{V(T_q)}(v)
\ge \chi(A_q)$.  If $u \not \in V(A_q)$, then we could move $u$ into $A_q$
without violating any hypotheses. So, we may assume that $q = 1$.  Since $T^*
\cup \set{v}$ does not induce a clique, there is some $A_p$ to which $v$ is not
joined. 

By hypothesis $d_{V(T_1)}(v) \ge \chi(A_1)$ and $T_1$ is complete, so
$v$ must be joined to $T_1$.  So, by considering only the indices $1, p, k$, we
can assume that $k=3$ and $p = 2$. More precisely, in what follows we will move
some vertices between parts $A_1$, $A_p$, and $A_q$ and color the graph
$H[V(A_1)\cup V(A_p)\cup V(A_q)]$ with at most $\chi(A_1)+\chi(A_p)+\chi(A_k)-1$
colors.  By combining this coloring with one that uses $\chi(A_i)$ colors on
each other part $A_i$, we show that $\chi(H)<\sum_{i=1}^k\chi(A_i)$.  This
contradiction proves Claim 2.

Pick $y \in V(T_2) \setminus N(v)$ and $z \in V(T_1 \setminus \{u\})$. Let $B_1
= G\brackets{(A_1 \cup \set{v,y}) \setminus \set{u,z}}$, $B_2 =
G\brackets{\parens{A_2 \cup \set{z}} \setminus \set{y}}$,  and $B_3 =
G\brackets{\parens{A_3 \cup \set{u}} \setminus \set{v}}$.  We derive a
contradiction by showing that $\chi(B_1) < \chi(A_1)$ and $\chi(B_2) \le
\chi(A_2)$ and $\chi(B_3) \le \chi(A_3)$.  

Since, $d_{A^*}(z) \le \card{T^*}$ and $T^*$ is complete, we have
$d_{V(A_2)}(z) \le \chi(A_2) + 1$ and hence $d_{V(B_2)}(z) = d_{V(A_2)}(z) - 1
\le \chi(A_2)$ since $z \adj y$.  Since $z$ has exactly $\chi(A_2) - 1$
neighbors in $T_2 \setminus\{y\}$, we see that $z$ has at most $\chi(A_2) - 1$
neighbors in each component of $B_2 \setminus\{z\}$ and hence $\chi(B_2) \le
\chi(A_2)$.  Since, by assumption, $d_{V(A_k)}(u) \le \chi(A_k) + 1$ and $T_k$
is complete, the proof that $\chi(B_3) \le \chi(A_3)$ is nearly identical (if
$u$ has a neighbor in $A_k\setminus T_k$, then we may need to permute colors on
its component, so that this neighbor does not use the same color as $u$).

Since $\set{u, z}$ is joined to $\set{v,y}$, we have $d_{V(B_1)}(v) =
d_{V(A_1)}(v) - 2 \le \chi(A_1) + 1 - 2 = \chi(A_1) - 1$. Similarly,
$d_{V(B_1)}(y) \le \chi(A_1) - 1$.  Let $K = G\brackets{T_1 \cup \set{v,y}
\setminus \set{u,z}}$. Then $K$ is a copy of
$K_{\chi(A_1)}$ with the edge $vy$ deleted.  First, color $B_1 \setminus V(K)$
with $\chi(A_1) - 1$ colors. Since $v$ and $y$ each have at most one neighbor
outside of $K$ in $B_1$ and $\chi(A_1) \ge 4$, we can finish the coloring on
$K$ by choosing the same color for $v$ and $y$, different from the colors used
on their at most $2$ (collective) neighbors in $B_1 \setminus V(K)$, and then
coloring $K \setminus\{v, y\}$ with the $\chi(A_1) - 2$ other colors (see
Figure \ref{HelperPrecursorFigure}).
\end{proof}
\setcounter{claim}{0}

In proving our next few lemmas, we repeatedly use the following helper lemma,
which is an easy corollary of Lemma~\ref{HelperPrecursor}.

\begin{lemma}
\label{helper}
Let $P$ be a minimum $(r_1, \ldots, r_s)$-partition of a graph $G$ with
$\chi(G) = \Delta(G) = 1 + \sum_{i \in \irange{s}} r_i$.  Let \[\S = ((P^1,v_1,
i_1, P^2),\ldots, (P^q, v_q, i_q, P^{q+1}))\] be a move sequence starting at
$P$. Let $R$ and $S$ be distinct full clubs of $\S$ and $t \in \irange{q + 1}$.  If $R_t
= \A(P^t)$, then
\begin{enumerate}
\item[(a)] if $u\in V(R_t)$ and $u$ has at least 2 neighbors in $S_t$, then $u$
is joined to $S_t$.
\item[(b)] if $u\in V(R_t)$ and $v\in V(S_t)$ and $u$ has at least 2 neighbors
in $S_t$ and $v$ has at least 2 neighbors in $R_t\setminus\{u\}$, then 
$v$ is joined to $R_t$.
\end{enumerate}
\end{lemma}
\begin{proof}
First we prove (a).
By symmetry, assume that $V(R_t)\subseteq P^t_1$ and $V(S_t)\subseteq
P^t_2$.  We apply Lemma \ref{HelperPrecursor} (a) with $A_i = G[P^t_i]$ for $i
\in \irange{2}$, $H = G[V(A_1) \cup V(A_2)]$ and $T_1 = R_t$.  By Lemma
\ref{GeneralMozhan}, $\chi(H) = r_1 + r_2 + 1 = \chi(A_1) + \chi(A_2)$ and
$\chi(A_1 - x) < \chi(A_1)$ for all $x \in V(T_1)$.  Also by Lemma
\ref{GeneralMozhan},  $d_{A_2}(x) \le \chi(A_2) + 1$ for all $x \in V(T_1)$.  By Lemma \ref{HelperPrecursor}, $u$ has at least $\chi(A_2)$ neighbors in some component $T_2$ of $A_2$.  Since $d_{A_2}(u) \le \chi(A_2) + 1$ and $u$ has at least two neighbors in $S_t$, we must have $T_2 = S_t$.  Since $S_t$ is a $K_{\chi(A_2)}$ this proves (a).

Now we prove (b).
If $d_{A_1}(v) > \chi(A_1) + 1$, then there exists some part $P_k^t$ with
$d_{P_k^t}(v) < r_k$.  By moving $v$ to $P_k^t$ and any vertex in $T_1$ to
$P^t_2$, we get a $(\chi(G) - 1)$-coloring of $G$, a contradiction.  So
$d_{A_1}(v)\le \chi(A_1)+1$.  By (a), $\card{N(u) \cap V(T_2)} = \chi(A_2)$ and
$v \in N(u) \cap V(T_2)$. So, we may apply Lemma \ref{HelperPrecursor} (b) to
conclude that $\card{N(v) \cap V(T_1)} \ge \chi(A_1)$.  Since $T_1$ is a
$K_{\chi(A_1)}$ this proves (b). 
\end{proof}

\begin{lemma}\label{ClubJoinLemma}
Let $P$ be a minimum $(r_1, \ldots, r_s)$-partition of a graph $G$ with
$\chi(G) = \Delta(G) = 1 + \sum_{i \in \irange{s}} r_i$.  Let $\S$ be a move
sequence starting at $P$ and let $R$ and $S$ be distinct full clubs of $\S$. Then, for
any $t_1, t_2 \ge 1$, we have that $R_{t_1}$ is joined to $S_{t_1}$ if and only if
$R_{t_2}$ is joined to $S_{t_2}$.
\end{lemma}
\begin{proof}
Suppose the lemma is false and let 
\[\S = ((P^1, v_1, i_1, P^2), \ldots, (P^q,v_q, i_q, P^{q+1}))\] 
\noindent be the shortest move sequence for which it fails.  There
must be a $t \in \irange{q}$ such that either $R_{t}$ is not joined to $S_{t}$,
but $R_{t + 1}$ is joined to $S_{t + 1}$ or else $R_{t}$ is joined to $S_{t}$,
but $R_{t + 1}$ is not joined to $S_{t + 1}$.  
Note that $q=1$, for otherwise, the move sequence $((P^t,v_t,i_t,P^{t+1}))$ is a
shorter counterexample.
Hence $\S = ((P^1, v_1, i_1, P^2))$.  Since the reverse sequence $(P^2, v_1,
j(P^1), P^1)$ is also a counterexample, we may assume that $R_1$ is not joined
to $S_1$, but $R_2$ is joined to $S_2$.

By symmetry between $R$ and $S$, we may assume that $R_1$ is the active
component.  Since $R_1$ is not joined to $S_1$, but $R_2$ is joined to $S_2$,
it must be that $R_2 = R_1 \setminus \{v_1\}$ is joined to $S_2 = S_1$ and
there is $u \in V(S_1)$ with $v_1 \nonadj u$.  
Pick $w \in V(R_1 \setminus \{v_1\})$.
Now applying Lemma~\ref{helper}(b) to $w$ and $u$ shows that $S_1$ is joined to $R_1$, a contradiction.
\end{proof}
\setcounter{claim}{0}

Lemma \ref{ClubJoinLemma} makes it possible for us to talk about full clubs being joined or not joined as follows.
\begin{defn}
\label{JoinedDefn}
Let $P$ be a minimum $(r_1, \ldots, r_s)$-partition of a graph $G$ with
$\chi(G) = \Delta(G) = 1 + \sum_{i \in \irange{s}} r_i$.  Let $\S$ be a move
sequence starting at $P$ and let $R$ and $S$ be distinct full clubs of $\S$.  Then $R$ and $S$ are \emph{joined} if $R_t$ and $S_t$ are joined for all $t \ge 1$.  Also $R$ and $S$ are \emph{not joined} if $R_t$ and $S_t$ are not joined for all $t \ge 1$.  Note that by Lemma \ref{ClubJoinLemma} $R$ and $S$ are either joined or not joined.
\end{defn}

\begin{defn}
Let $P$ be a minimum $(r_1, \ldots, r_s)$-partition of a graph $G$.  For a club $R$ of a move sequence 
\[\S = ((P^1, v_1, i_1, P^2), \ldots, (P^q, v_2, i_2, P^{q+1}))\]
\noindent starting at $P$, we say that \emph{$R$ is active $k$ times} if the number of $t \in \irange{q+1}$ such that $R_i$ is active is $k$.
\end{defn}

\begin{lemma}\label{JoiningIsTransitive}
	Let $G$ be a graph with $\chi(G) = \Delta(G) = 1 + \sum_{i \in \irange{s}} r_i$ and let $\S$ be a move
	sequence starting at a minimum $(r_1, \ldots, r_s)$-partition of $G$. If $S$ is a full club of $\S$ that is active at least once and
 $R$ and $W$ are different full clubs of $\S$ such that $R$ is joined to $S$ and $S$ is joined to $W$, then $R$ is joined to $W$.
\end{lemma}
\begin{proof}
Pick $t$ such that $S_t$ is active and say the $t$-th move of $\S$ is $(P, v_t, i_t, P')$.  Put $T_1 = S_t$, $T_2 = R_t$, and
$T_3 = W_t$.  By symmetry, we assume that $V(T_1) \subseteq P_1$, $V(T_2)
\subseteq P_2$, and $V(T_3) \subseteq P_3$.  
We will apply Lemma \ref{HelperPrecursor} with $A_i = G[P_i]$ for all $i \in
\irange{3}$ and $H = G[V(A_1) \cup V(A_2) \cup V(A_3)]$.  
Define $A^*$ and $T^*$ as in Lemma~\ref{HelperPrecursor}.

Pick $u \in V(T_1)$.   We have $\chi(H) = r_1 + r_2 +
r_3 + 1 = \chi(A_1) + \chi(A_2) + \chi(A_3)$ and $\chi(A_1 \setminus \{u\}) <
\chi(A_1)$.  Also by Lemma \ref{GeneralMozhan},  $d_{V(A_3)}(u) \le
\chi(A_3) + 1$.  Since $T_3$ is a $K_{r_3}$, we also have $d_{V(T_3)}(u) =
\chi(A_3)$.  For any
$v \in V(T_3)$, we have $d_{A^*}(v) \le 1 + \chi(A_1) + \chi(A_2)$, for
otherwise there exists some part $P_q$ with $d_{P_q}(v) < r_q$.  By moving $v$
to $P_q$ and $u$ to $P_3$, we get a $(\chi(G) - 1)$-coloring of $G$, a
contradiction.  Also, $d_{T^*}(v) \ge 3$ since $T_1$ is joined to $T_3$. 
Additionally, $T^*$ induces a clique and $T_k$ is complete.  To apply Lemma
\ref{HelperPrecursor}, it remains to check that  $d_{A^*}(w) \le \card{T^*}$
for all $w \in T^*$.  If not, then we could move $w$ to some part $P_q$ with
$d_{P_q}(w) < r_q$ and get a  $(\chi(G) - 1)$-coloring of $G$.
So, we apply Lemma \ref{HelperPrecursor} with each $v \in V(T_3)$ and conclude that $T_3$ is joined to $T_2$ as desired.
\end{proof}
\setcounter{claim}{0}

\begin{defn}
Let $P$ be a minimum $(r_1, \ldots, r_s)$-partition of a graph $G$.  For a club
$R$ of a move sequence $\S$ starting at $P$, the \emph{spread} of $R$ is the set
of indices of parts to which $R$ sends vertices; more formally,
\[\sp_{\S}(R) = \setbs{i}{(Q, v, i, Q') \in \S \text{ with } \C(\A(Q)) = R}.\]
\noindent The \emph{spread} of $\S$ is $\sp(\S) = \max_R \card{\sp(R)}$ where
the max is over all clubs $R$ of $\S$.
\end{defn}

\begin{lemma}\label{ActiveAtMostThreeTimes}
Let $P$ be a minimum $(r_1, \ldots, r_s)$-partition of a graph $G$ with $\chi(G) = \Delta(G) = 1 + \sum_{i \in \irange{s}} r_i$.  If
\[\S = ((P^1, v_1, i_1, P^2), \ldots, (P^q, v_q, i_q, P^{q+1}))\] 
\noindent is a move sequence with $\sp(\S) \le 2$, then one of the following holds:
\begin{enumerate}
\item[(1)] $v_i = v_j$ for some distinct $i, j \in \irange{q}$ (i.e.~a vertex
moves more than once); or
\item[(2)] there is $t \in \irange{q}$ such that the active component in $P^t$ is joined to the active component in $P^{t+1}$; or
\item[(3)] every club of $\S$ is active at most $3$ times.
\end{enumerate}
\end{lemma}
\begin{proof}
Suppose the lemma is false and choose a move sequence
\[\S = ((P^1, v_1, i_1, P^2), \ldots, (P^q, v_q, i_q, P^{q+1}))\] 
\noindent for which it fails minimizing $q$.   By minimality of $q$, we have a
length three subsequence $((P^a, v_a, i_a, P^{a+1}), (P^b, v_b, i_b, P^{b+1}),
(P^{c-1}, v_{c-1}, i_{c-1}, P^c))$ of $\S$ such that 
\begin{enumerate}
\item[(i)] $\C^a(\A(P^a)) = \C^b(\A(P^b)) = \C^c(\A(P^{c}))$ and $\C^{a+1}(\A(P^{a+1}))
= \C^{b+1}(\A(P^{b+1}))$; and
\item[(ii)] there is at most one $(P^d, v_d, i_d, P^{d+1})$ in $\S$ with $a < d
< b$ such that $\C^d(\A(P^d)) = \C^a(\A(P^a))$; and
\item[(iii)] $\C^{a+1}(\A(P^{a+1}))$ is active at most $3$ times.
\end{enumerate}
	
Let $X = \C(\A(P^a))$ and $Y = \C(\A(P^{a+1}))$.  
We may choose $c$ to be the smallest index in $\{b+1,\ldots,q+1\}$ such that $X$
is active at stage $c$.
We will show that $X$ is joined to
$Y$, which gives a contradiction, since we are assuming (2) does not hold.   
If there does not exist $(P^d, v_d, i_d, P^{d+1})$ in $\S$ with $a < d < b$
such that $\C(\A(P^d)) = \C(\A(P^a))$, then let $d = b$.

\begin{center}
%
\begin{figure}[ht]
\begin{center}
\begin{tikzpicture}[scale = .8]
\tikzset{
  treenode/.style = {align=center, inner sep=0pt, text centered,
    font=\sffamily, minimum size=20pt},
  norm/.style = {treenode, circle, black, draw=black, 
    text width=1.5em},
  active/.style = {treenode, circle, black, draw=black, 
    text width=1.5em,line width=.2em},
  moved/.style = {treenode, circle, black, draw=black, 
    text width=1.5em,very thick},
}
\tikzstyle{node}= [circle, black, draw=black, text width=1.5em]
\draw[white] (0,3.75) node{};

\def\a{1.6}
\def\short{1.7in}
\def\longer{2.2in}
\draw (0,0) node[moved]{\tiny {$x_1$}};
\draw (0,1) node[norm]{\tiny {$v_b$}};
\draw (0,2) node[norm]{\tiny {$v_d$}};
\draw (0,3) node[active]{\tiny {$v_1$}};
\draw (\a,0) node[norm]{\tiny {$y_1$}};
\draw (\a,1) node[norm]{\tiny {$v_{b+1}$}};
\draw (\a,2) node[norm]{\tiny {$v_2$}};

\draw (0,-1) node{$X_1$};
\draw (\a,-1) node{$Y_1$};

\begin{scope}[xshift=\short]
\draw (0,0) node[norm]{\tiny {$x_1$}};
\draw (0,1) node[norm]{\tiny {$v_b$}};
\draw (0,2) node[norm]{\tiny {$v_d$}};
\draw (\a,0) node[norm]{\tiny {$y_1$}};
\draw (\a,1) node[norm]{\tiny {$v_{b+1}$}};
\draw (\a,2) node[active]{\tiny {$v_2$}};
\draw (\a,3) node[moved]{\tiny {$v_1$}};

\draw (0,-1) node{$X_2$};
\draw (\a,-1) node{$Y_2$};
\draw (-1.3,1.5) node{$\huge{\Rightarrow}$};

\begin{scope}[xshift=\longer]
\draw (0,0) node[norm]{\tiny {$x_1$}};
\draw (0,1) node[norm]{\tiny {$v_b$}};
\draw (0,2) node[active]{\tiny {$v_d$}};
\draw (0,3) node[moved]{\tiny {$x_2$}};
\draw (\a,0) node[norm]{\tiny {$y'_1$}};
\draw (\a,1) node[norm]{\tiny {$v_{b+1}$}};
\draw (\a,2) node[norm]{\tiny {$v_1$}};
\draw (0,-1) node{$X_d$};
\draw (\a,-1) node{$Y_d$};
\draw (-1.45,1.5) node{$\huge{\Longrightarrow}$};
\end{scope}
\end{scope}
\begin{scope}[yshift=-2.4in]

\draw (0,0) node[norm]{\tiny {$x_1$}};
\draw (0,1) node[active]{\tiny {$v_b$}};
\draw (0,2) node[norm]{\tiny {$x_2$}};
\draw (0,3) node[moved]{\tiny {$x_3$}};
\draw (\a,0) node[norm]{\tiny {$y''_1$}};
\draw (\a,1) node[norm]{\tiny {$v_{b+1}$}};
\draw (\a,2) node[norm]{\tiny {$v_1$}};
%
\draw (0,-1) node{$X_b$};
\draw (\a,-1) node{$Y_b$};
\draw (-1.45,1.5) node{$\huge{\Longrightarrow}$};

\begin{scope}[xshift=\short]
\draw (0,0) node[norm]{\tiny {$x_1$}};
\draw (0,1) node[norm]{\tiny {$x_2$}};
\draw (0,2) node[norm]{\tiny {$x_3$}};
\draw (\a,0) node[norm]{\tiny {$y''_1$}};
\draw (\a,1) node[active]{\tiny {$v_{b+1}$}};
\draw (\a,2) node[norm]{\tiny {$v_1$}};
\draw (\a,3) node[moved]{\tiny {$v_b$}};
%
\draw (0,-1) node{$X_{b+1}$};
\draw (\a,-1) node{$Y_{b+1}$};
\draw (-1.3,1.5) node{$\huge{\Rightarrow}$};

\begin{scope}[xshift=\longer]
\draw (0,0) node[norm]{\tiny {$x_1$}};
\draw (0,1) node[norm]{\tiny {$x_2$}};
\draw (0,2) node[norm]{\tiny {$x_3$}};
\draw (0,3) node[moved]{\tiny {$x_4$}};
\draw (\a,0) node[norm]{\tiny {$y'''_1$}};
\draw (\a,1) node[norm]{\tiny {$v_1$}};
\draw (\a,2) node[norm]{\tiny {$v_b$}};
\draw (0,-1) node{$X_c$};
\draw (\a,-1) node{$Y_c$};
\draw (-1.45,1.5) node{$\huge{\Longrightarrow}$};
\draw[white] (0,-1.5) node{};
\end{scope}
\end{scope}
\end{scope}

\end{tikzpicture}
\end{center}
\captionsetup{width=0.8\textwidth}
\caption{The six key partitions $X_i,Y_i$ in the proof of
Lemma~\ref{ActiveAtMostThreeTimes}.  In each partition, the next vertex that
will move is marked in bold, and the vertex that most recently moved is marked
in semi-bold.  If a vertex is unnamed in the proof, we denote it as $x_i$ or
$y_i$ based on whether it appears in ${X_j \mbox{ or }Y_j}$.}

\end{figure}
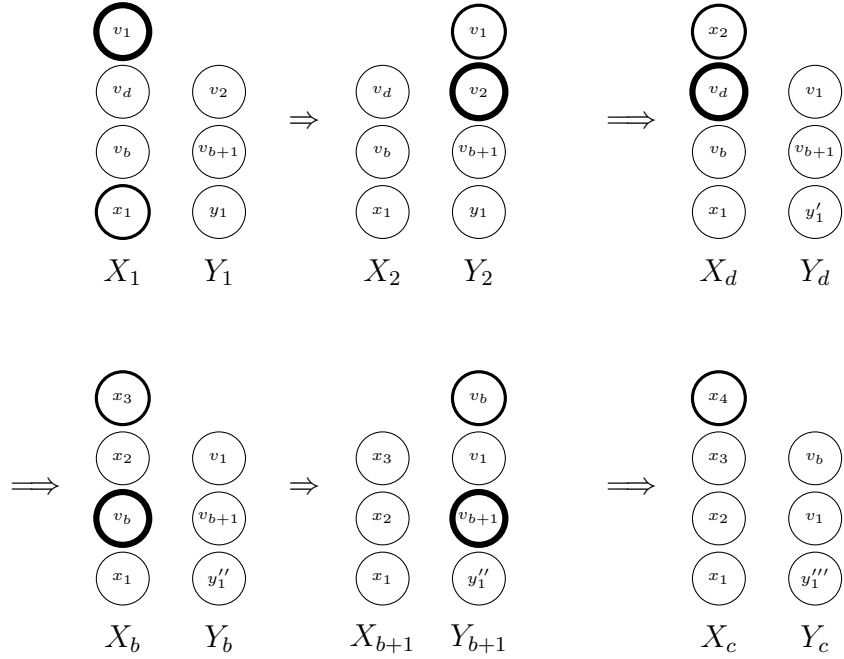

\end{center}

\begin{claim}
$\set{v_b}$ is joined to $V(Y_d)$.
\label{claim1}
\end{claim}
\noindent
Since $Y$ becomes active at most once (by (iii)) between move $d$ and move
$b+1$, we have $\card{V(Y_d) \cap V(Y_b)} \ge 2$.  One vertex in this
intersection is $v_a$, and another is $v_{b+1}$ (since no vertex is moved
twice, by (1)).  So $v_b$ is adjacent to $v_a$ and $v_{b+1}$, since
$v_a,v_b,v_{b+1}\in V(Y_{b+1})$ and $Y$ is full.  Applying Lemma~\ref{helper}(a) to $X$ and $Y$ with $t=d$, shows that
$v_b$ is joined to $V(Y_d)$, since $v_a,v_{b+1}\in V(Y_d)$.

\begin{claim}
$\set{v_a}$ is joined to $V(X_d)$.
\label{claim2}
\end{claim}

\noindent
Since $\card{V(X_d)\cap V(X_a)}\ge 3$, $v_a$ has at least 3 neighbors in $X_d$.
Now we show that $v_a$ is joined to $V(X_d)$ by 
Lemma~\ref{helper}(b).  Specifically, we apply the lemma to $X$ and $Y$ with $t=d$.
We let $u=v_b$ and $v=v_a$.  Claim 1 states that $v_b$ is joined to $Y_d$, so,
in particular, $v_b$ has at least two neighbors in $V(Y_d)$.  Since $X_a$ is
full, $v_a$ is adjacent to both $v_d$ and $x_a$ (the final vertex moved before
$\S$ began).  So Lemma~\ref{helper}(b)
implies that $\set{v_a}$ is joined to $V(X_d)$, as desired.
\begin{claim}
$\set{v_a}$ is joined to $V(X_b)$.
\label{claim3}
\end{claim}
\noindent
Since $Y$ is full, $v_b$ is joined to $V(Y_b)$.  Since $\card{V(X_d)\cap
V(X_b)}\ge 3$ and $v_a$ is joined to $V(X_d)$, $v_a$ has at least 3 neighbors
in $X_b$.  So $v_a$ is joined to $V(X_b)$ by Lemma~\ref{helper}(b) applied to $X$ and $Y$ with $t=b$.

\begin{claim}
$V(X_{b+1})$ is joined to $V(Y_c)$.
\label{claim4}
\end{claim}
\noindent
Since $V(X_{b+1})\subset V(X_b)$, Claim 3 shows that $\set{v_a, v_b}$ is joined to $V(X_{b+1})$.  But, $\set{v_a, v_b} \subset V(Y_c)$, so applying Lemma~\ref{helper}(a) to $X$ and $Y$ 
with $t=c$ shows that $V(X_{b+1})$ is joined to $V(Y_c)$.

\begin{claim}
$V(X_c)$ is joined to $V(Y_c)$. In particular, $X$ is joined to $Y$.
\label{claim5}
\end{claim}

\noindent
Since $\card{X_{b+1}} \ge 3$, Claim~\ref{claim4} and an application of and Lemma~\ref{helper}(b) to $X$ and $Y$ 
with $t=c$ shows that $V(X_{c})$ is joined to $V(Y_c)$.
\end{proof}
\setcounter{claim}{0}

\setcounter{lemma}{2}
\begin{thm}
If $G$ is a graph with $\chi(G) \ge \Delta(G)$, then $\omega(G) \ge
\Delta(G) - 3$ if $\Delta(G) \equiv 1\pmod 3$ and $\omega(G) \ge \Delta(G) -
4$ otherwise.
\label{firstthm}
\end{thm}
\setcounter{lemma}{9}
\begin{proof}
The theorem is trivially true if $\Delta(G)\le 6$, so we assume that
$\Delta(G)\ge 7$.
It suffices to consider critical graphs, since any graph $G$ contains a
critical subgraph $H$ with $\chi(H)=\chi(G)$.
By Brooks' Theorem, we may assume $\chi(G) = \Delta(G)$.  Let $s =
\Floor{\frac{\Delta(G) - 1}{3}}$ and $r_1, \ldots, r_s \in \set{3,4}$ such that
$\Delta(G) = 1 + \sum_{i \in \irange{s}} r_i$.  Now $G$ has an $(r_1, \ldots,
r_s)$-partition, so we can let $P$ be a minimum $(r_1, \ldots, r_s)$-partition
of $G$.  Let 
\[\S = ((P^1, v_1, i_1, P^2), \ldots, (P^q, v_q, i_q, P^{q+1}))\] 
\noindent be a move sequence starting at $P$ with $\sp(\S) \le 2$ of maximum
length such that $v_i \ne v_j$ for all pairs of distinct $i, j \in
\irange{q}$ and for each $t \in \irange{q}$ the active component in $P^t$ is
not joined to the active component in $P^{t+1}$.  Let $A = \A(P^{q+1})$.  
Lemma \ref{ActiveAtMostThreeTimes} implies that $\C^{q+1}(A)$ is active at most $3$
times in $\S$.  Since $r_{i_q} \ge 3$, there is $x \in V(A)$ such that $x \not
\in \setbs{v_t}{t \in \irange{q}}$, i.e., $x$ has never moved during $\S$.  

Let $T = \sp(\C(A))$.  If there is $i \in T$ with $d_{P^{q+1}_i}(x) = r_i$,
then we have a move $(P^{q+1}, x, i, Q^{i})$ and by the maximality condition on
$\S$, it must be that $A$ is joined to $\A(Q^{i})$.  But, by assumption, $A$ is
not joined to $\A(Q^{i})$ for any $i \in T$, so this is impossible.  

Since $d_G(x) \le 1 + \sum_{i \in \irange{s}} r_i$ and $x$ has exactly
$r_{i_q}$ neighbors in $P^{q+1}_{i_q}$, there is at most one $i \in \irange{s}
\setminus \set{i_q}$ for which $d_{P^{q+1}_i}(x) > r_i$.  So, $\card{T} \le 1$
and if $\card{T} = 1$, then $T$ contains the one $i$ with $d_{P^{q+1}_i}(x) >
r_i$.  By the maximality condition on $\S$, it must be that $A$ is joined to
clubs in $P^{q+1}_i$ for all but one $i \in \irange{s} \setminus \set{i_q}$. 
By Lemma~\ref{GeneralMozhan}(2), we know that each club joined to $A$ is full.
By Lemma \ref{JoiningIsTransitive}, all of these full clubs must be pairwise
joined to each other.  Thus, together, they induce a large clique.
Specifically, they induce a clique of size $1+\sum_{j\in
\irange{s}\setminus\set{i}}r_j$, which is size $\Delta(G)-r_{i}$.
Since $r_j = 3$ if $\Delta(G) \equiv 1\pmod 3$ and $r_j \le 4$
otherwise, we have the desired large clique.
\end{proof}

As an immediate consequence of
Theorem~\ref{firstthm}, we get the following corollary.

\begin{cor}
\label{base13}
If $G$ is a graph with $\chi(G)\ge \Delta(G)= 13$, then $G$ contains $K_{10}$.
\end{cor}

\section{The First Main Theorem}
\label{sectionMain1}

A \emph{hitting set} is an independent set that intersects every maximum
clique.  If $I$ is a hitting set and also a maximal independent set, then
$\Delta(G-I)\le \Delta(G)-1$ and $\chi(G-I)\ge \chi(G)-1$.  (In our
applications, we can typically assume that $\Delta(G-I)=\Delta(G)-1$, since
otherwise we get a good coloring or a big clique from Brooks' Theorem.  We
give more details in the proof of Theorem~\ref{main1}.) So if 
$G-I$ has a clique of size $\Delta(G-I)-t$, for some constant $t$, then
also $G$ has a clique of size $\Delta(G)-t$.
We repeatedly remove hitting sets to reduce a graph with $\Delta\ge 13$
to one with $\Delta=13$.  Since we proved in Corollary~\ref{base13}
that every graph with $\chi\ge\Delta=13$ contains $K_{10}$, this repeated
removal of hitting sets allows us to prove that every $G$ with
$\chi\ge\Delta\ge13$ contains $K_{\Delta-3}$.

This idea is not new.  Kostochka~\cite{Kostochka} proved
that every graph with $\omega\ge\Delta-\sqrt{\Delta}+\frac32$ has a hitting set.
Rabern~\cite{Rabern} extended this result to the case $\omega\ge\frac34(\Delta+1)$,
and King~\cite{King} strengthened his argument to prove 
that $G$ has a hitting set if $\omega>\frac23(\Delta+1)$.  This condition is
optimal, as illustrated by the lexicographic product of an odd cycle and a clique.
Finally, King's argument was refined by Christofides, Edwards, and
King~\cite{CEK} to show that these lexicographic products of odd cycles and
cliques are the only sharpness examples; that is, $G$ has a hitting set if
$\omega\ge\frac23(\Delta+1)$ and $G$ is not the lexicographic product of an odd
cycle and a clique.  Hitting set reductions have application to other vertex
coloring problems. 
Using this idea (and others), King and Reed~\cite{KR} gave a
short proof that there exists $\epsilon > 0$ such that
$\chi\le\ceil{(1-\epsilon)(\Delta+1)+\epsilon\omega}$. 

To keep this paper largely self-contained, we prove our own hitting set lemma. 
In the present context, it suffices to find a hitting set when $G$ is a minimal
counterexample to Theorem~\ref{main1} with $\Delta\ge14$. 
Such a $G$ is $\Delta$-critical, which facilitates a shorter proof.
In~\cite{CR1}, we proved a number of results about so called $d_1$-choosable
graphs (defined below), which are certain graphs that cannot appear as induced
subgraphs in a $\Delta$-critical graph.  We leverage these $d_1$-choosability
results to prove our hitting set lemma, then we use the hitting set lemma to
reduce the problem to the case $\Delta=13$, which we handled in
Corollary~\ref{base13}.  Since the proofs of the $d_1$-choosability results
in~\cite{CR1} are lengthy, we give a short proof of the special case that we
need here.

A \emph{list assignment} $L$ is an assignment $L(v)$ of a set of allowable
colors to each vertex $v\in V(G)$.  An $L$-coloring is a proper coloring such
that each vertex $v$ is colored from $L(v)$.  An $f$-assignment is a list
assignment $L$ such that $\card{L(v)}=f(v)$ for all $v\in V(G)$.
In particular, a $d_1$-assignment is an $f$-assignment with $f(v)=d(v)-1$ for
all $v$.
A graph $G$ is \emph{$f$-choosable} if $G$ has an $L$-coloring for every 
$f$-assignment $L$.  
No $\Delta$-critical graph contains an induced $d_1$-choosable
subgraph $H$ (by criticality, color $G\setminus H$, then extend the coloring to
$H$, since it is $d_1$-choosable).  For a list assignment $L$, let
$Pot(L)=\cup_{v\in V(G)}L(v)$.  The following lemma is central in proving each
of our $d_1$-choosability results.

\begin{lemma}[Small Pot Lemma, \cite{Kierstead,RS}]
\label{SPL}
For a list size function $f:V(G) \rightarrow \{0,\ldots,\card{G}-1\}$, a graph
$G$ is $f$-choosable if and only if $G$ is $L$-colorable for each list
assignment $L$ such that $\card{L(v)}=f(v)$ for all $v\in V(G)$ and
$\card{\cup_{v\in V(G)}L(v)} < \card{G}$.
\end{lemma}

\begin{proof}
Fix $G$ and $f$, and let $V=V(G)$.  The ``only if'' direction is true by definition.
Now we prove the ``if'' direction.  Assume that $G$ is $L$-colorable for 
each list assignment $L$ such that $\card{L(v)}=f(v)$ for all $v$ and
$\card{\cup_{v\in V}L(v)}<\card{G}$.  For any $U \subseteq V$ and any
list assignment $L$, let $L(U)$ denote $\cup_{v\in U}L(v)$.  Let $L$ be an
$f$-assignment such that $|L(V)| \ge |G|$ and $G$ is not $L$-colorable.  For
each $U \subseteq V$, let $g(U) = |U| - |L(U)|$.  Let $\B$ be a bipartite
graph, where one part consists of vertices in $V$ and the other part
consists of colors in $Pot(L)$, and a vertex $v$ is adjacent to a color $c$ if
$c\in L(v)$.  Since $G$ is not $L$-colorable, $\B$ has no matching saturating
$V$, so Hall's Theorem implies there exists $U$ with $g(U)>0$.  Choose
$U$ to maximize $g(U)$.  Let $A$ be an arbitrary set of $|G|-1$ colors
containing $L(U)$.  Construct $L'$ as follows.  For $v\in U$, let $L'(v)=L(v)$.
 Otherwise, let $L'(v)$ be an arbitrary subset of $A$ of size $f(v)$.  Now
$|L'(V)| < |G|$, so by hypothesis, $G$ has an $L'$-coloring.  This gives an
$L$-coloring of $U$.  By the maximality of $g(U)$, for all $W \subseteq
(V\setminus U)$, we have $|L(W)\setminus L(U)| \ge |W|$.  
Let $\B'=\B\setminus (\cup_{u\in U} \{u\}\cup N_{\B}(u))$.
Thus, by Hall's Theorem, $\B'$ has a matching saturating $V\setminus U$; 
so we can extend the $L$-coloring of $U$ to all of $V$.
\end{proof}

\begin{lemma}[\cite{CR1}]
\label{K_tClassification}
For $t \geq 4$, $\join{K_t}{B}$ is not $d_1$-choosable if and only if
$\omega(B) \geq \card{B} - 1$; or $t = 4$ and $B$ is $E_3$ or $K_{1,3}$; or $t
= 5$ and $B$ is $E_3$.
\end{lemma}
\begin{proof}
If $\omega(B)\ge |B|-1$, then assign each $v\in V(\join{K_t}{B})$ a subset of
$\{1,\ldots,t+|B|-2\}$; since $\omega(\join{K_t}{B})\ge t+|B|-1$, clearly $G$
is not colorable from this list assignment.  Now let $G=\join{K_5}{E_3}$, and
note that $\join{K_4}{K_{1,3}}\cong \join{K_5}{E_3}$.  Consider the
following list assignment $L$ for $G$: each dominating vertex has list
$\{1,\ldots,6\}$ and the three other vertices get distinct lists among
$\{1,2,3,4\}, \{1,2,5,6\}, \{3,4,5,6\}$.  If $G$ has a proper
$L$-coloring, then the dominating vertices use five distinct colors;
this leaves only one color for the three remaining vertices, but no color
appears in all three lists.  Hence, $G$ has no $L$-coloring.
Now form $G'$ from $G$  by deleting one dominating vertex (note that $G'=
\join{K_4}{E_3}$), and let $L'=L\setminus\{6\}$. Since $G$ has no $L$-coloring,
also $G'$ has no $L'$-coloring.  This proves one direction of the lemma;
now we consider the other.

Suppose the only if direction of the lemma is false, and let $G$ and $L$ be a
minimal counterexample, where $G=\join{K_t}{B}$ and $L$ is a $d_1$-assignment. 
Since $\omega(B)\le|B|-2$, subgraph $B$ contains either (i) an independent set
$S=\{x_1,x_2,x_3\}$ or  (ii) a set $S=\{x_1,x_2,x_3,x_4\}$ with
$x_1x_2,x_3x_4\notin E(B)$.  If $B$ contains only (i), then $G[S] = E_3$ and $t
\ge 6$ (by moving any dominating vertices from $B$ to $K_t$).
Let $T=V(K_t)$ and denote $T$ by $\{y_1,\ldots,y_t\}$.  
In Cases (i) and (ii) we assume by minimality that $t=6$ and $t=4$,
respectively (for larger $t$, we can greedily color all but 6 (resp. 4)
vertices; each vertex has enough colors to be colored greedily, since at least
4 of its neighbors remain uncolored).  
Also by minimality, we assume that $V(B)=S$ (as for larger $t$, we can greedily
color vertices of $B$ not in $S$).

By definition $\card{L(v)}= d(v)-1$; specifically,
$|L(x_i)|= d_S(x_i)+t-1$ and $|L(y_j)|= |S|+t-2$ for all $x_i\in S$ and
$y_j\in T$.  
When we have $i, j, k$ with $x_i\nonadj x_j$ and $|L(x_i)|+|L(x_j)|>|L(y_k)|$,
we often use the following technique, called \emph{saving a color} on $y_k$ via
$x_i$ and $x_j$.  
If there exists $c\in L(x_i)\cap L(x_j)$, then use $c$ on $x_i$ and $x_j$. 
Otherwise, color just one of $x_i$ and $x_j$ with some $c\in (L(x_i)\cup
L(x_j))\setminus L(y_k)$.  For a set $U$, let $L(U)=\cup_{v\in U}L(v)$.

Case (i)  
By the Small Pot Lemma, assume that $|L(G)|\le 8$.  This implies $|L(x_i)\cap
L(x_j)|\ge 2$ for all $i,j\in[3]$.  If there exist $x_i$ and $y_k$
with $L(x_i)\not\subseteq L(y_k)$, then color $x_i$ to save a color on
$y_k$.  Color the remaining $x$'s with a common color; this saves an
additional color on each $y$.  Now finish greedily, ending with $y_k$.
Thus, we have $L(x_i)\subset L(y_k)$ for all $i\in [3]$ and $k\in [6]$.  This
gives $\card{\cup_{i=1}^3L(x_i)}\le 7$.  Since
$\sum_{i=1}^3|L(x_i)|=15>2|\cup_{k=1}^3L(x_k)|$, we have a color $c\in
\cap_{i=1}^3L(x_i)$.  Use $c$ on every $x_i$ and finish greedily.

Case (ii)
By the Small Pot Lemma, assume that $|L(G)|\le 7$.  If $S$ induces at least two edges, then
$|L(x_1)|+|L(x_2)|\ge 8$.
So $L(x_1)\cap L(x_2)\ne \emptyset$.
Color $x_1$ and $x_2$ with a
common color $c$.  If $|L(y_1)\setminus\{c\}|\le 5$, then save a color on $y_1$
via $x_3$ and $x_4$.  Now finish greedily, ending with $y_1$.

Suppose $S$ induces exactly one edge; by symmetry, say it is $x_1x_3$.  
Suppose that 
$L(x_1)\cap L(x_2)\ne\emptyset$. 
Similar to the previous argument, use a common color on $x_1$ and $x_2$,
possibly save on $y_1$ via $x_3$ and $x_4$, then finish greedily.  
So instead, assume that $L(x_1)\cap L(x_2)=\emptyset$. 
Since $\card{L(G)}\le 7$ and $L(x_1)\cap L(x_2)=\emptyset$,
by symmetry (between $x_1$ and $x_3$ and also between $x_2$ and $x_4$), we may
assume that $L(x_1)=L(x_3)=\{a,b,c,d\}$ and $L(x_2)=L(x_4)=\{e,f,g\}$.  Also by
symmetry, $a$ or $e$ is missing from $L(y_1)$.  So color $x_1$ with $a$ and
$x_2$ and $x_4$ with $e$ and $x_3$ arbitrarily; this saves one color on each
$y_i$ and a second color on $y_1$.  Now finish greedily, ending with $y_1$.  

So instead $G[S]=E_4$. 
If a common color appears on 3 vertices of $S$, use it there, then finish
greedily.  If not, then by pigeonhole, at least 5 colors appear on pairs of
vertices; so, two colors appear on disjoint pairs.  Color two such disjoint
pairs, each with a common color.  Now finish the coloring greedily.
\end{proof}

The following lemma of King enables us to find an independent transversal.

\begin{lemma}[Lopsided Transversal Lemma~\cite{King}]\label{Lopsided}
Let $H$ be a graph and $V_1 \cup \cdots \cup V_r$ a partition of $V(H)$.  
If there exists $s \geq 1$ such that for each $i \in \irange{r}$ and each
$v \in V_i$ we have $d(v) \leq \min\set{s, \card{V_i}-s}$, then $H$ has an
independent transversal $I$ of $V_1, \ldots, V_r$.
\end{lemma}

Now we have all the tools to prove our hitting set lemma.

\begin{lemma}
\label{ourHittingLemma}
Every $\Delta$-critical graph with $\chi \geq \Delta \geq 14$ and
$\omega=\Delta-4$ has a hitting set.
\end{lemma}
\begin{proof}
Suppose the lemma is false, and let $G$ be a counterexample minimizing $\card{G}$.  
Consider distinct
intersecting maximum cliques $A$ and $B$.  Since a vertex in their intersection
has degree at most $\Delta$, we have $\card{A \cap B} \geq \card{A} + \card{B}
- (\Delta + 1) = \Delta - 9 \geq 5$.  Since $G$ contains no induced
$d_1$-choosable subgraph, letting $A\cap B=K_t$ in Lemma
\ref{K_tClassification} implies that $\omega(G[A\cup B])\ge|A\cup B|-1$. 
Hence $\card{A \cap B} = \omega-1 = \Delta - 5$.
Suppose $C$ is another maximum clique intersecting $A$;
let $U = A \cup B \cup C$ and $J = A\cap B \cap C$.   We use inclusion-exclusion
to bound $\card{U}$ and $\card{J}$.
First, $\card{U}=\card{A\cup B\cup C}=\card{A\cup B}+\card{C\setminus(A\cup B)}
\le \card{A\cup B}+\card{C\setminus A}=\card{A\cup B}+\card{C}-\card{C\cap A} \le (\Delta-5+1+1)+(\Delta-4)-(\Delta-5)=\Delta-2$.
Second, $\card{J} = \card{A\cap B}+\card{C} - \card{(A\cap B)\cup C} \ge
\card{A\cap B}+\card{C}-\card{U} \ge (\Delta - 5) + (\Delta - 4) - (\Delta-2)=
\Delta - 7\ge 7$.  

Since $\card{J}\ge 7$, by Lemma \ref{K_tClassification},
$\omega(G[U]) \geq \card{U} - 1$; so $C=A$ or $C=B$, a contradiction.  
Thus, every maximum clique intersects at most one other maximum
clique.  Hence we can partition the union of the maximum cliques into sets
$D_1, \ldots, D_r$ such that either $D_i$ is a $(\Delta-4)$-clique $C_i$ or
$D_i = C_i \cup \set{x_i}$ for a $(\Delta-4)$-clique $C_i$, where $x_i$ is
adjacent to all but one vertex of $C_i$.

For each $D_i$, if $D_i=C_i$, then let $K_i=C_i$.  If $D_i=C_i\cup\{x_i\}$, then
let $K_i=C_i\cap N(x_i)$.
Consider the subgraph $F$ of $G$ formed by taking the subgraph induced on the
union of the $K_i$ and then making each $K_i$ independent.  We
apply Lemma \ref{Lopsided} to $F$ with $s = \frac{\Delta}{2} - 2$.  We have two
cases to check, when $K_i = C_i$ and when $K_i = C_i \cap N(x_i)$.  In the
former case, $\card{K_i} = \Delta-4$ and for each $v \in K_i$ we have $d_F(v)
\leq \Delta(G) + 1 - (\Delta-4) = 5$.  Hence $d_F(v) \leq \frac{\Delta}{2} - 2
= \min\set{\frac{\Delta}{2} - 2, \Delta - 4 - (\frac{\Delta}{2} - 2)}$ since
$\Delta \geq 14$.  In the latter case, we have $\card{K_i} = \Delta - 5$ and
since every $v \in K_i$ is adjacent to $x_i$ and to the vertex in $C_i\setminus
K_i$, neither of which is in $F$, we have $d_F(v) \leq \Delta - (\Delta-4) =
4$.  This gives  $d_F(v) \leq \frac{\Delta}{2} - 3 = \min\set{\frac{\Delta}{2}
- 2, \Delta-5 - (\frac{\Delta}{2} - 2)}$ since $\Delta \geq 14$.  Now Lemma
\ref{Lopsided} gives an independent transversal $I$ of the $K_i$, which is a
hitting set.  
\end{proof}

Now we can prove the first of our two main results.  For convenience, we restate it.

\begin{named1}
Every graph with $\chi \geq \Delta \geq 13$ contains $K_{\Delta-3}$.
\end{named1}
\begin{proof}
Let $G$ be a counterexample minimizing $|G|$; note that $G$ is vertex critical.
By Corollary~\ref{base13}, we have $\Delta(G) \ge 14$.  
By Theorem~\ref{firstthm}, we know that $\omega(G)\ge \Delta(G)-4$.
Since $G$ contains no $K_{\Delta-3}$, we know that $\omega(G)=\Delta(G)-4$.
So let $I$ be a hitting set given
by Lemma~\ref{ourHittingLemma}, expanded to a maximal independent set.
Now $\omega(G-I)<\Delta(G) - 4$, $\Delta(G-I)\le \Delta(G)-1$, and $\chi(G-I)
\ge \chi(G)-1$.  If $\Delta(G-I) \le \Delta(G)-3$, then greedy coloring gives
$\chi(G-I) \le \Delta(G-I)+1 \le \Delta(G)-2$, so $\chi(G)\le \Delta(G)-1$.
If $\Delta(G-I) = \Delta(G)-2$, then $\chi(G-I)\le\Delta(G-I)$ by Brooks'
Theorem (since $\omega(G-I)<\Delta(G)-4$), so $\chi(G) \le \Delta(G)-1$.
So instead $\Delta(G-I) = \Delta(G) - 1$.
Now $\chi(G-I) \ge \Delta(G-I) \geq 13$ and $\omega(G-I)< \Delta(G-I) - 3$,
contradicting the minimality of $|G|$.
\end{proof}

We suspect that Theorem \ref{main1} holds for all $\Delta$.  By Theorem \ref{firstthm} and Theorem \ref{main1}, the following conjecture is only open when $\Delta \in \set{6,8,9,11,12}$.

\begin{conj}
Every graph with $\chi \ge \Delta$ contains $K_{\Delta-3}$.
\end{conj}

We conclude this section with a nice application of Theorem \ref{main1} to the
Borodin-Kostochka conjecture for vertex-transitive graphs.  Suppose $G$ is a
vertex-transitive graph with $\chi(G) \ge \Delta(G) \ge 15$.  Now $\omega(G)
\ge \Delta(G) - 3$ by Theorem \ref{main1}.  Since $G$ is vertex-transitive,
every vertex of $G$ is in a $K_{\Delta(G) - 3}$.  In \cite{denseneighborhoods},
it was proved that the Borodin-Kostochka conjecture holds for graphs where
every vertex is in a $K_{\frac23\Delta(G) + 2}$.  Now $\Delta(G) - 3 \ge
\frac23\Delta(G) + 2$ since $\Delta(G) \ge 15$, so we have proved the following.

\begin{cor}\label{BKTransitive}
Every vertex-transitive graph with $\chi \ge \Delta \ge 15$ contains $K_\Delta$.
\end{cor}

Corollary \ref{BKTransitive} should hold for $\Delta \ge 9$ and this may be much easier to prove than the full Borodin-Kostochka conjecture.  In a short note \cite{transitivenote}, we explore these ideas further and prove  Corollary \ref{BKTransitive} for $\Delta \ge 13$.   A more general conjecture comes out of these considerations which is worth mentioning because it implies Corollary \ref{BKTransitive} for $\Delta \ge 9$.

\begin{conj}
Every vertex-transitive graph satisfies $\chi \le \max \set{\omega, \ceil{\frac{5\Delta + 3}{6}}}$.
\end{conj}

\section{The Second Main Theorem}

In this section, we prove our second main theorem.  First, we prove a
lemma that follows from~\cite{CR1} about list coloring (we use it to
forbid a certain subgraph in a $\Delta$-critical graph). 

%
%

\begin{lemma}[\cite{CR1}]\label{mixedLemmaK3}
Let $G=\join{K_3}{E_2}$. If $L$ is a list assignment such that $\card{L(v)}
\geq d(v) - 1$ for all $v \in V(G)$ and for some $w\in V(K_3)$ and
some $x\in V(E_2)$ we have $\card{L(w)} \geq d(w)$ and $\card{L(x)}\ge d(x)$,
then $G$ has an $L$-coloring.
\end{lemma}
\begin{proof}
Denote $V(E_2)$ by $\set{x, y}$. 
By the Small Pot Lemma, we assume $\card{Pot(L)} \leq 4<5\le \card{L(x)} + \card{L(y)}$.  
After coloring $x$ and $y$ the same, finish greedily, ending with $w$.
\end{proof}

In the rest of this section, we extend and refine the ideas in
Section~\ref{mozhan-partitions}.

\begin{defn}
Let $P$ be a minimum $(r_1, \ldots, r_s)$-partition of a graph $G$ with $\chi(G) = \Delta(G) = 1 + \sum_{i \in \irange{s}} r_i$.  Let $\S$ be a move sequence starting at $P$.
For a full club $S$ with respect to $\S$, the \emph{clubgroup} $\G_{\S}(S)$ of $S$ is the set consisting of $S$ and the full clubs to which $S$ is joined. 
\end{defn}

When the move sequence is clear from context, we write $\G(S)$ in place of
$\G_{\S}(S)$. Clearly if $R$ and $S$ are full clubs and $R \in \G(S)$, then $S
\in \G(R)$. By Lemma \ref{JoiningIsTransitive}, we know that if $R$, $S$, and
$T$ are full clubs, and $R \in \G(S)$ and $S \in \G(T)$, then $R \in \G(T)$. 
So, the set of full clubs with respect to $\S$ is partitioned into clubgroups.
We need a way of differentiating moves that are internal to a clubgroup and
moves that go from one clubgroup to another.  This motivates the following
definition of \emph{internal} and \emph{external} moves.

With the notation we have at this point, referring to objects like ``the clubgroup of the club of the active component'' is a bit unwieldy.  
So, we allow ourselves to write $\G_\S(A)$ in place of $\G_\S(\C_S(A))$.  

\begin{defn}
Let $P$ be a minimum $(r_1, \ldots, r_s)$-partition of a graph $G$ with $\chi(G)
= \Delta(G) = 1 + \sum_{i \in \irange{s}} r_i$.  Let $\S$ be a move sequence
starting at $P$.  Let $M = (P^a, v, i, P^b)$ be a move in $\S$, $A^a$ the active
component in $P^a$ and $A^b$ the active component in $P^b$.  Then move $M$ is
\emph{internal} if $\G_{\S}(A^a) = \G_{\S}(A^b)$. Otherwise,
$M$ is \emph{external}.  We write $\E(\S)$ for the subsequence of $\S$
consisting of all the external moves of $\S$.
\end{defn}

\begin{defn}
Let $P$ be a minimum $(r_1, \ldots, r_s)$-partition of a graph $G$ with
$\chi(G) = \Delta(G) = 1 + \sum_{i \in \irange{s}} r_i$.  Let $\S = ((P^1, v_1,
i_1, P^2), \ldots, (P^q, v_q, i_q, P^{q+1}))$ be a move sequence starting at
$P$.  Let $R$ be a full club of $\S$.  We say that the clubgroup $\G_\S(R)$ is
\emph{activated at least $k$ times} if there is a subsequence $((P^{a_1},
v_{a_1}, i_{a_1}, P^{a_1 + 1}), \ldots, (P^{a_k}, v_{a_k}, i_{a_k}, P^{a_k +
1})$ of $\E(\S)$ where the active club in $P^{a_i + 1}$ is in $\G_\S(R)$ for $i
\in \irange{k}$. 
\end{defn}

\begin{defn}
Let $P$ be a minimum $(r_1, \ldots, r_s)$-partition of a graph $G$ with
$\chi(G) = \Delta(G) = 1 + \sum_{i \in \irange{s}} r_i$.  Let $\S = ((P^1, v_1,
i_1, P^2), \ldots, (P^q, v_q, i_q, P^{q+1}))$ be a move sequence starting at
$P$. Let $R$ be a full club of $\S$. The \emph{external spread} of $R$
is \[\esp_{\S}(R) = \setbs{i}{(Q, v, i, Q') \in \E(\S) \text{ with }
\C^i(\A(Q))\in \G_\S(R)}.\]

\noindent The \emph{external spread} of $\S$ is $\esp(\S) = \max_R \card{\esp(R)}$ where the max is over all full clubs $R$ of $\S$.
\end{defn}

In an $(r_1, \ldots, r_s)$-partition of a graph $G$ a clubgroup containing $s - 1$ clubs is called a \emph{big} clubgroup.  A clubgroup with fewer than $s-1$ clubs is \emph{small}.
Our next key lemma is Lemma~\ref{ClubgroupActivation}, which is an analogue of
Lemma~\ref{ActiveAtMostThreeTimes}.
Intuitively, it says that clubgroups can be thought of much like clubs: in a
move sequence with external spread at most 2 (and each vertex moved at most
once), each clubgroup is activated at most 3 times.  The proof is similar to
that of Lemma~\ref{ActiveAtMostThreeTimes}.  Not suprisingly, we must first
prove an analogue of the helper lemma that played a key role in that proof.
This is Lemma~\ref{helperGen} which follows quickly from Lemma
\ref{HelperPrecursor}.

\begin{lemma}
\label{helperGen}
Let $P$ be a minimum $(r_1, \ldots, r_s)$-partition of a graph $G$ with
$\chi(G) = \Delta(G) = 1 + \sum_{i \in \irange{s}} r_i$.  Let \[\S = ((P^1,v_1,
i_1, P^2),\ldots, (P^q, v_q, i_q, P^{q+1}))\] be a move sequence starting at
$P$. Let $R$ and $S$ be full clubs of $\S$ and $t \in \irange{q + 1}$.  If $R_t
= \A(P^t)$, then
\begin{enumerate}
\item[(a)] if $u\in V(R_t)$ and $u$ has at least 2 neighbors in $S_t$, then $u$
is joined to $S_t$.
\item[(b)] if $u\in V(R_t)$ and $v\in V(S_t)$ and $u$ has at least 2 neighbors
in $S_t$ and $v$ has at least 2 neighbors in $V(\G(R_t))\setminus\{u\}$, then 
$v$ is joined to $V(\G(R_t))$.
\end{enumerate}
\end{lemma}

\begin{proof}
(a) is the same as (a) in Lemma~\ref{helper}; we only restate it here for
convenience.

(b): By symmetry, we may assume that $V(\G(R_t))$ intersects each of $P_1^t,
\ldots, P_{k-1}^t$ and none of $P_k^t, \ldots, P_s^t$.  Moreover, we assume
that $V(S_t) \subseteq P_k^t$.
Let $A_i = G\brackets{P_i^t}$ for $i \in \irange{k}$.  Let $H =
G\brackets{V(\set{A_1, \ldots, A_k})}$ and let $T_1$ be the component of $A_1$
containing $u$.  Plainly, $\chi(H) = \sum_{i \in \irange{k}} \chi(A_i)$.  
By Lemma \ref{GeneralMozhan}, $\chi(A_1 \setminus\{u\}) < \chi(A_1)$ and
$d_{A_k}(u) \le \chi(A_k) + 1$.    By Lemma \ref{HelperPrecursor} (a), vertex
$u$ has at least $\chi(A_k)$ neighbors in some component $T_k$ of $A_k$.  Since
$d_{A_k}(u) \le \chi(A_k) + 1$ and $u$ has at least two neighbors in $S_t$, we
must have $T_k = S_t$.

If $d_{A^*}(v) > 1 + \sum_{i \in \irange{k - 1}} \chi(A_i)$, then there exists
some part $P_q^t$ with $d_{P_q^t}(v) < r_q$.  By moving $v$ to $P_q^t$ and $u$
to $P^t_k$, we get a $(\chi(G) - 1)$-coloring of $G$, a contradiction.  
So $d_{A^*}(v)\le 1+\sum_{i\in \irange{k-1}}\chi(A_i)\le \card{T^*}$.
Similarly, $d_{A^*}(w)\le \card{T^*}$ for all $w\in T^*$.
To finish the proof of (b), we now apply Lemma \ref{HelperPrecursor} (c), with
$T^*=V(\G(R_t))$.
\end{proof}

\begin{lemma}
\label{ClubgroupActivation}
Let $P$ be a minimum $(r_1, \ldots, r_s)$-partition of a graph $G$ with $\chi(G) = \Delta(G) = 1 + \sum_{i \in \irange{s}} r_i$.   If
\[\S = ((P^1, v_1, i_1, P^2), \ldots, (P^q, v_q, i_q, P^{q+1}))\] is a move
sequence with $\esp(\S) \le 2$ and $v_i \ne v_j$ for all distinct $i, j \in
 \irange{q+1}$, then:
\begin{enumerate}
\item[(1)] every clubgroup of $\S$ is activated at most $3$ times; and
\item[(2)] every big clubgroup of $\S$ is activated at most $2$ times.
\end{enumerate}
\end{lemma}
\begin{proof}
Suppose the lemma is false and choose a move sequence
\[\S = ((P^1, v_1, i_1, P^2), \ldots, (P^q, v_q, i_q, P^{q+1}))\] 
\noindent for which it fails minimizing $q$.   By minimality of $q$ (and since
$\esp(\S)\le 2$), we have a length three subsequence $((P^a, v_a, i_a, P^{a+1}),
(P^b, v_b, i_b, P^{b+1}), (P^{c-1}, v_{c-1}, i_{c-1}, P^{c}))$ of $\S$ such that 
	\begin{enumerate}
		\item[(i)] $\G(\A(P^a)) = \G(\A(P^b)) = \G(\A(P^{c}))$ and
$\C^{a+1}(\A(P^{a+1})) = \C^{b+1}(\A(P^{b+1}))$; and
		\item[(ii)] there is at most one $(P^d, v_d, i_d, P^{d+1})$ in $\S$ with $a < d < b$ such that $\G(\A(P^d)) = \G(\A(P^a))$; and
		\item[(iii)] $\C^{a+1}(\A(P^{a+1}))$ is active at most $3$ times.
	\end{enumerate}
	
Let $X = \G(\A(P^a))$ and $Y = \C^{a+1}(\A(P^{a+1}))$.  We will show that $X$ is joined to
$Y$; this gives a contradiction, since we are assuming $Y$ is not in the
clubgroup of $X$.   
We may choose $c$ to be the smallest index in $\{b+1,...,q+1\}$ such that $X$
is active at stage $c$.
If there does not exist $(P^d, v_d, i_d, P^{d+1})$ in $\S$ with $a < d < b$ such
that $\C^d(\A(P^d)) = \C^a(\A(P^a))$, then let $d = b$.
The proof of (1) is nearly identical to the proof of
Lemma~\ref{ActiveAtMostThreeTimes}.  The only difference is that each
instance of Lemma~\ref{helper} in that proof is now replaced by
Lemma~\ref{helperGen}; so we omit the proof.

Now for the proof of (2).
 If a clubgroup is big, then each of its external moves goes to the same
part $X_i$ of the partition.  Thus, if a big clubgroup becomes active 3
times, then we again have the move subsequence $((P^a, v_a, i_a, P^{a+1}),
(P^b, v_b, i_b, P^{b+1}), (P^{c-1}, v_{c-1}, i_{c-1}, P^{c}))$, with properties (i), (ii),
and (iii) above.  Hence, the proof of (1) is also valid in this context, and
yields a proof of (2).
\end{proof}

Now we can prove our second main theorem (we restate it for convenience), which 
strengthens Theorem~\ref{krsone} for $\Delta \ge 10$.

\begin{thm}[Kostochka, Rabern, and Stiebitz \cite{KRS}]\label{krsone}
If $G$ is a critical graph with $\chi(G) \ge \Delta(G)$ and $\omega(G) <
\Delta(G)$, then 
$\omega(\H) \ge \Floor{\frac{\Delta(G) - 1}{2}}$.
\end{thm}

\begin{named2}
If $G$ is a graph with $\chi(G) \ge \Delta(G)$ and $\omega(G) < \Delta(G)$,
then $\omega(\H) \ge \Delta(G) - 4$ if $\Delta(G) \equiv 1\pmod 3$ and
$\omega(\H) \ge \Delta(G) - 5$ otherwise.
\end{named2}
\begin{proof}
Suppose the theorem is false and let $G$ be a minimal counterexample.  Note 
that $G$ is a critical graph with $\chi(G) \ge \Delta(G)$, $\omega(G) <
\Delta(G)$, and $\omega(\H) < \Delta(G) - 4$ if $\Delta(G) \equiv 1\pmod 3$ and
$\omega(\H) < \Delta(G) - 5$ otherwise.  By Brooks' Theorem, we have $\chi(G) =
\Delta(G)$. By Theorem \ref{krsone}, $\Delta(G) \ge 10$.

Let $s = \Floor{\frac{\Delta(G) - 1}{3}}$ and $r_1, \ldots, r_s \in \set{3,4}$
such that $\Delta(G) = 1 + \sum_{i \in \irange{s}} r_i$.  Since $G$ has an
$(r_1, \ldots, r_s)$-partition, we can let $P$ be a minimum $(r_1, \ldots,
r_s)$-partition of $G$.  Let $\S = ((P^1, v_1, i_1, P^2), \ldots, (P^q, v_q,
i_q, P^{q+1}))$ be a move sequence starting at $P$ that never moves a low vertex
within a clubgroup, with $\esp(\S) \le 2$,
with $v_i \ne v_j$ for all distinct $i, j \in \irange{q+1}$, and,
subject to that, $P$ has the maximum number of external moves.
Let $A = \A(P^{q+1})$.  Hereafter $\G(\C(A))$ denotes $\G(\C^{q+1}(A))$.

Suppose $\G(\C(A))$ is small. By Lemma \ref{ClubgroupActivation},
$\G(\C(A))$ is activated at most $3$ times in $\S$.  Since $r_{i_q} \ge 3$,
there is $x \in V(A)$ such that $x \not \in \setbs{v_t}{t \in \irange{q}}$,
i.e., since $A$ has at least 4 vertices, some $x\in V(A)$ has not yet moved.
Since $\G(\C(A))$ is small, there exist at least two indices
$j_1,j_2\in\irange{s}$ such that $\G(\C(A))$ has no club in part $j_1$ and no
club in part $j_2$. Now for some $i\in\{1,2\}$, we have $d_{P_{j_i}}(x)\le
r_{j_i}$.  By Lemma~\ref{GeneralMozhan}(2), we know that 
$x$ is joined to a full club in part $j_i$, so moving $x$ to part
$j_i$ is a valid move.  

We show that in all cases, we can extend the move sequence $\S$ to a sequence
$\S'$ by moving $x$ to part $j_i$;
we need only to verify that after moving $x$ to part $j_i$, the resulting
sequence $\S'$ satisfies  $\esp(\S') \le 2$.
If we presently have $\esp(\S)\le 1$, then clearly $\esp(\S') \le 2$.  If
instead the present sequence has  $\esp(\S) \le 2$,
then we can choose $j_1$ and $j_2$ such that $\S$ contains external
moves to both.  In that case, moving $x$ to one of the parts will not increase
the external spread.  So in all cases, we contradict the maximality of the move
sequence.


Hence $\G(\C(A))$ is big.  By Lemma \ref{ClubgroupActivation},
$\G(\C(A))$ is activated at most $2$ times in $\S$.  Consider $K = \bigcup_{Z
\in \G(\C(A))} V(Z_{q+1})$.  Since $\G(\C(A))$ is big, $K$ is a clique that has
vertices in all but one part of $P^{q+1}$.  By renumbering if necessary, we may
assume that $K$ has vertices in each of $P^{q+1}_1, \ldots, P^{q+1}_{s-1}$; 
so $\card{K} = 1 + \sum_{i \in \irange{s-1}} r_i$.  Hence $\card{K} =
\Delta(G) - 3$ if $\Delta(G) \equiv 1 \pmod 3$ and $\card{K} \ge
\Delta(G) - 4$ otherwise.  In either case, $K$ has at least two low vertices,
since $G$ is a counterexample to the theorem.

If $K$ contains a low vertex $x$ that has not moved, i.e., $x \in K \setminus
\setbs{v_t}{t \in \irange{q}}$,
then we can move $x$ to part $s$ (by Lemma~\ref{GeneralMozhan}(2)), which
contradicts the fact that $\S$ has maximum length.
Hence, $K$ does not contain an unmoved low vertex.  
Since low vertices are not moved within clubgroups, and each low vertex in $K$
has already moved, we know that each was moved externally.
So, since $\G(\C(A))$ is
activated at most $2$ times in $\S$, we know that $K$ has exactly two low
vertices, $v$ and $w$.  Since both $v$ and $w$ have moved, $\G(\C(A))$ is activated exactly $2$ times; 
one time when $v$ was moved in and one time when $w$ was moved in.
Therefore, $\S$ contains external moves $(P^{a_1}, v,
i_{a_1}, P^{a_1 + 1})$ and $(P^{a_2}, w, i_{a_2}, P^{a_2 + 1})$ and in both
$P^{a_1 + 1}$ and $P^{a_2 + 1}$  the clubgroup $\G(\C(A))$ contains the active
club (possibly different each time).  
By symmetry, assume $a_1 < a_2$ and so
$a_2 = q$.

Let $B$ be the active component in $P^{q}$.  Since $w \in V(B)$ and $w$ is
adjacent to at least $\Delta(G) - 5$ vertices in $K$, we see that $\C(B)$'s
clubgroup is $\set{\C(B)}$ (otherwise $w$ would be adjacent to
more than 5 vertices coming from $\C(B)$'s clubgroup, which is too many).
Suppose that $V(B)$ contains a high vertex that is unmoved, i.e., $z \in V(B)
\setminus \setbs{v_t}{t \in \irange{q-1}}$.  Since $\Delta(G) \ge 10$,
we have $s \ge 3$.  So there is an external move $M = (P^q, z, i, Q)$ where $i
\in \irange{s-1}$.  Consider the move sequence formed from $\S$ by removing the
last move and appending $M$.  By our considerations in the previous paragraph,
this move sequence can be extended (the active club now contains an unmoved low
vertex, since the last vertex moved is high), contradicting the maximality
condition on $\S$.  So, every $z \in V(B) \setminus \setbs{v_t}{t \in
\irange{q-1}}$ is low.  

Since $w$ is low, for every move $(Q, z, i, Q')$ in
$\S$ where $\C(B)$ is active in $Q$, we must have $z \in K$;  otherwise $w$
would have at least $\Delta$ neighbors.
In particular, there are at most two such moves since $\G(\C(A))$ is activated
at most twice.  So $B$ contains an unmoved vertex, i.e., $\card{V(B) \setminus
\setbs{v_t}{t \in \irange{q}}} \ge 1$.

Let $(P^{a_3}, u, i_{a_3}, P^{a_3 + 1})$ be the first external move in $\S$
after $(P^{a_1}, v, i_{a_1}, P^{a_1 + 1})$.  Let $A'$ be the active component
in $P^{a_3}$ and consider $K' = \bigcup_{Z \in \G(\C(A'))} V(Z_{a_3})$.  
Since $\card{K'}=\card{K}$, as we saw before for $K$, also $K'$ has at least
two low vertices $v, w'$.  If $u$ is high, then $K$ would contain low vertices
$v, w, w'$, a contradiction.  So $u$ is low; in fact, $u=w'$. 

We show that  $\C(\A(P^{a_3 + 1})) = \C(B)$. Since $v$ is low, we have the move
$M'=(P^{a_3}, v, s, Q')$. Let $B' = \A(Q') \setminus \set{v}$. Since $v$ is
adjacent to $w$ (and $v$ is low), we must have $w \in V(B')$.  So $\C(B) =
\C(B')$.  Since $\C(B')$ is active at most twice, $v$ has at least $\card{B'} -
2 > 0$ neighbors in $\C(B')_{q+1}$.  Since $v$ is low, we have the move $M =
(P^{q+1}, v, s, Q)$.  Now Lemma \ref{GeneralMozhan}, part (2) shows that
$\set{v} \cup V(\C(B')_{q+1})$ induces a $K_{r_s + 1}$.  But $u \in P^{q+1}_s$
and $v$ is adjacent to $u$, so $u \in V(\C(B')_{q+1})$.  Therefore,
$\C(\A(P^{a_3 + 1})) = \C(B') = \C(B)$.
 
Now we have the $K_3$ on $\set{u,v,w}$ joined to a set of vertices $T$ with
$\card{T} = \Delta(G) - 3$.  Namely, $T=(V(K)\setminus\{v,w\})\cup
(V(B)\setminus\{u\})$.
Moreover, since $\card{V(B) \setminus
\setbs{v_t}{t \in \irange{q}}} \ge 1$, there is a low vertex in
$V(B\setminus\{v_q,u\})$ and $V(B\setminus\{v_q,u\}) \subseteq T$.  So, by
Lemma \ref{mixedLemmaK3}, $\set{u,v,w} \cup T$ induces a $K_{\Delta(G)}$, a
contradiction.
\end{proof}

We conjecture that the previous theorem actually holds with $\omega(\H) \geq \Delta
- 5$ replaced by $\omega(\H) \geq \Delta - 4$.  
In \cite{Rabern2}, the second author proved this result for $\Delta = 6$; 
later in \cite{KRS} it was proved for $\Delta=7$.  The condition
$\omega(\H) \geq \Delta - 4$ would be tight since the graph $O_5$ in Figure
\ref{fig:O5} is a counterexample to  $\omega(\H) \geq \Delta - 3$ when
$\Delta=5$.  In fact, it was shown in \cite{KRS} that $O_5$ is the only
counterexample to $\omega(\H) \geq \Delta - 3$ when $\Delta=5$.

\begin{conj}
Let $G$ be a graph. 
If $\chi\ge\Delta$, then $\omega\ge\Delta$ or $\omega(\H)\ge \Delta-4$.
\end{conj}

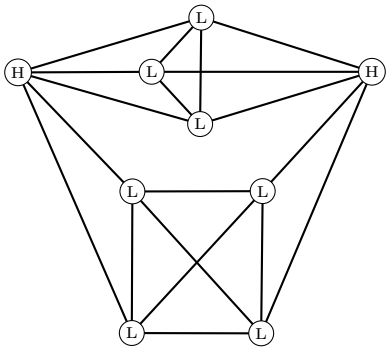
\begin{figure}[hbt]
\centering
\begin{tikzpicture}[scale = 10]
\tikzstyle{VertexStyle}=[shape = circle, minimum size = 1pt, inner sep = 1.2pt, draw]
\Vertex[x = 0.270266681909561, y = 0.890800006687641, L = \tiny {L}]{v0}
\Vertex[x = 0.336266696453094, y = 0.962799992412329, L = \tiny {L}]{v1}
\Vertex[x = 0.334666579961777, y = 0.821199983358383, L = \tiny {L}]{v2}
\Vertex[x = 0.56306654214859, y = 0.890800006687641, L = \tiny {H}]{v3}
\Vertex[x = 0.244666695594788, y = 0.731600046157837, L = \tiny {L}]{v4}
\Vertex[x = 0.417866677045822, y = 0.732000052928925, L = \tiny {L}]{v5}
\Vertex[x = 0.243866696953773, y = 0.543200016021729, L = \tiny {L}]{v6}
\Vertex[x = 0.415866762399673, y = 0.542800068855286, L = \tiny {L}]{v7}
\Vertex[x = 0.0926666706800461, y = 0.890000000596046, L = \tiny {H}]{v8}
\tikzstyle{EdgeStyle}=[]
\Edge[](v1)(v0)
\tikzstyle{EdgeStyle}=[]
\Edge[](v2)(v0)
\tikzstyle{EdgeStyle}=[]
\Edge[](v3)(v0)
\tikzstyle{EdgeStyle}=[]
\Edge[](v2)(v1)
\tikzstyle{EdgeStyle}=[]
\Edge[](v3)(v1)
\tikzstyle{EdgeStyle}=[]
\Edge[](v2)(v3)
\tikzstyle{EdgeStyle}=[]
\Edge[](v5)(v4)
\tikzstyle{EdgeStyle}=[]
\Edge[](v6)(v4)
\tikzstyle{EdgeStyle}=[]
\Edge[](v6)(v5)
\tikzstyle{EdgeStyle}=[]
\Edge[](v6)(v7)
\tikzstyle{EdgeStyle}=[]
\Edge[](v7)(v4)
\tikzstyle{EdgeStyle}=[]
\Edge[](v7)(v5)
\tikzstyle{EdgeStyle}=[]
\Edge[](v4)(v8)
\tikzstyle{EdgeStyle}=[]
\Edge[](v6)(v8)
\tikzstyle{EdgeStyle}=[]
\Edge[](v5)(v3)
\tikzstyle{EdgeStyle}=[]
\Edge[](v7)(v3)
\tikzstyle{EdgeStyle}=[]
\Edge[](v0)(v8)
\tikzstyle{EdgeStyle}=[]
\Edge[](v1)(v8)
\tikzstyle{EdgeStyle}=[]
\Edge[](v2)(v8)
\end{tikzpicture}
\captionsetup{width=0.8\textwidth}
\caption{The graph $O_5$ is a $\Delta$-critical graph with $\Delta=5$ and $\omega(\H)=1$.}
\label{fig:O5}
\end{figure}

\section{Algorithms}
All of our coloring proofs do translate into algorithms to construct the
colorings.  However these algorithms cannot obviously be made to run in
polynomial time.  Attempts to do so encounter two main obstacles.  The first
comes in our proof of Theorem~\ref{firstthm}, when we consider a critical
subgraph $H$ of our given graph $G$.  We do not know an efficient
algorithm to find such a critical subgraph; however, we will see how to overcome
this difficulty.  Our second obstacle comes from King's Lopsided Transversal
Lemma.  While his proof is constructive, the algorithm it implies may require
exponential time.  We are not aware of any workaround to efficiently find our
hitting set; however, when $\Delta$ is sufficiently large,
we can use an idea of Alon instead.  We implement a modified version of the algorithm from
Theorem~\ref{firstthm}.

\begin{thm}\label{FirstAlgorithm}
There is a $\mathcal{O}(V^2 E^2)$ time graph algorithm that finds either a $(\Delta-1)$-coloring or a clique on $\Delta-4$ vertices ($\Delta-3$ vertices if
$\Delta\equiv 1\pmod 3$).
\end{thm}
\begin{proof}
Let $G$ be an $n$-vertex graph with $\Delta\ge 10$, and let $I$ be a maximal
independent set in $G$.  Let $G_0 = G-I$, and note that $\Delta(G_0) \le
\Delta(G)-1$.  Lov{\'a}sz's proof of Brooks' theorem \cite{lovasz1975three} can be implemented in time $\mathcal{O}(V + E)$ (see \cite{baetz2001brooks}).  
Applying this to $G_0$ we either get a $\Delta(G)$-clique or a $(\Delta(G) -
1)$-coloring of $G_0$.  In the former case, we are done, so suppose we have a
$(\Delta(G) - 1)$-coloring $\phi$ of $G_0$.

Let $v$ be an arbitrary vertex in $I$ and put $G_1 = G[V(G_0)\cup\{v\}]$.  We give an algorithm that either finds a $(\Delta(G)-1)$-coloring of $G_1$ or a clique on $\Delta(G)-4$ vertices ($\Delta(G)-3$ vertices if
$\Delta(G)\equiv 1\pmod 3$). Iterating this gives the desired algorithm.

Note that $G_1$ has an $(r_1,\ldots,r_s)$-partition $P$, where $s=\Floor{\frac{\Delta(G)-1}3}$ and
$r_1,\ldots,r_s\in\{3,4\}$; choose an arbitrary such partition which respects
the color classes of $\phi$.  Now we will construct a move sequence as in the
proof of Theorem~\ref{firstthm}, treating the resulting partitions as if they
were minimum partitions.  For each partition arising from the move sequence, we check whether any
property in Lemma~\ref{GeneralMozhan} is violated; if some property is violated
for a partition $P$, then we can modify $P$ to form a new partition $P'$ such
that $P'$ has fewer edge within parts, i.e., $\sigma(P')<\sigma(P)$.  When this
happens, we begin our move sequence anew, starting from $P'$.  Eventually, we
will reach a partition and a move sequence that does not allow us
to reduce the number of edges within parts.
Such a move sequence will terminate with either
(1) a clique on $\Delta(G)-4$ vertices ($\Delta(G)-3$ vertices if
$\Delta(G)\equiv 1\pmod 3$) or (2) a $(\Delta(G)-1)$-coloring of $G_1$. 
In the case of (1), our algorithm halts.  In the case of (2), we add a new
vertex $v'$ from $I\setminus\{v\}$ and continue.  

So, we need only analyze the running time.  Each move sequence has length at most $n$, since each vertex
moves at most once.  After adding a vertex, we can reduce the number of edges
within parts at most $\card{E(G)}$ times.
Hence, after we add a new vertex from $I$ to our partition, we need at most
$n\card{E(G)}$ moves until we find either a big clique or a
$(\Delta(G)-1)$-coloring.  After each move, we can verify that the resulting
partition satisfies all the properties of Lemma~\ref{GeneralMozhan} (or doesn't) and find a vertex to swap with
in $\mathcal{O}(V + E)$ time.  Since we need to do this at most $n|I|\card{E(G)}$ times, the running time of the algorithm is $\mathcal{O}(V^2 E^2)$.
\end{proof}

When $\Delta\not\equiv 1\pmod 3$, 
Theorem \ref{FirstAlgorithm} only finds a $K_{\Delta-4}$;
but Theorem \ref{main1} guarantees a $K_{\Delta-3}$ when $\Delta \ge 13$. To
get an algorithmic version of this result, we need to  efficiently find a
hitting set when $\chi=\Delta$ and
$\omega=\Delta-4$. We will show how to do this when $\Delta$ is sufficiently large. 
The proof we present here works for $\Delta\ge 37$. We also sketch how to refine this idea to work for $\Delta\ge 33$. Further, using a result of Kolipaka, Szegedy and Xu \cite{kolipaka2012sharper}, we show how to get down to $\Delta\ge 26$.  The general idea is to find a set of
disjoint cliques $\A=\{A_1,A_2,\ldots\}$ such that $\card{A_i}$ is large for
all $i$ and each maximum clique contains some $A_i$.  Following an idea of
Alon, we choose one vertex uniformly at random from each $A_i$ and use the
Lovasz Local Lemma to prove that with positive probability the chosen vertices
form an independent set.  Our proof uses one classical lemma each from
Hajnal~\cite{Hajnal} and Kostochka~\cite{Kostochka}.

\begin{lemma}[Hajnal~\cite{Hajnal}]
\label{hajnal}
If $\S$ is a collection of maximum cliques in a graph $G$, then
$$\card{\bigcup\S} + \card{\bigcap\S} \ge 2\omega.$$
\end{lemma}
\begin{proof}
We use induction on $\card{\S}$; the base case $\card{\S}=1$ is trivial.
Let $S_1\in \S$ and $\S'=\S-S_1$.
Consider the set $(\cap\S'\setminus S_1)\cup (S_1\cap (\cup \S'))$, which induces a clique.
Since $S_1$ is a maximum clique, $\card{S_1}\ge\card{(\cap\S'\setminus
S_1)\cup (S_1\cap (\cup \S'))}$, which yields $\card{S_1\setminus(\cup
\S')}\ge\card{(\cap \S')\setminus S_1}$.  By hypothesis, $\card{\cup
\S'}+\card{\cap \S'}\ge 2\omega$.  Adding this to the previous inequality
gives the desired result.
\end{proof}

Now we need the following definition.  Given a collection $\S$
of sets, the \emph{intersection graph} $X_{\S}$ has one vertex for each set of
$\S$ and two vertices are adjacent if their sets intersect.

\begin{lemma}[Kostochka~\cite{Kostochka}]
\label{kostochka}
Let $G$ be a graph with $\omega(G)>\frac23(\Delta(G)+1)$.
If $\S$ is a collection of maximum cliques in $G$ and the intersection graph
$X_{\S}$ is connected, then $\card{\bigcap\S}\ge 2\omega(G) - (\Delta(G)+1)$.
\end{lemma}
\begin{proof}
We use induction on $\card{\S}$; the base case $\card{S}=1$ is trivial.
The key is to show that $\card{\bigcap\S}>0$, for then $\card{\bigcup\S}\le
\Delta(G)+1$, so the lemma follows directly from Lemma~\ref{hajnal}.
Let $S_1\in\S$ be a noncutvertex of $X_{\S}$, and choose $S_2\in \S$ that
intersects $S_1$.  Lemma~\ref{hajnal} for the set $\{S_1,S_2\}$ implies
$\card{S_1\setminus S_2}=\card{S_1}-\card{S_1\cap S_2} \le
\omega(G)-(2\omega(G)-(\Delta(G)+1))=\Delta(G)+1-\omega(G)$.
Let $\S'=\S-S_1$.  Now $X_{\S'}$ is connected, so by
hypothesis, the lemma holds for $\S'$.  Choose $v\in \bigcap \S'$.  Now
$\card{\bigcup \S'}\le d_G(v)+1\le \Delta(G)+1$.  Thus, $\card{\bigcup \S}\le
\card{\bigcup \S'}+\card{S_1\setminus S_2}\le
(\Delta(G)+1)+(\Delta(G)+1-\omega(G))<2\omega(G)$.  By Lemma~\ref{hajnal}, 
$\card{\bigcap S}>0$, so the lemma follows.
\end{proof}

In \cite{Kostochka}, Kostochka used Lemma \ref{hajnal} and Lemma \ref{kostochka} to prove that a hitting set always exists when $\omega \ge \Delta + \frac32 - \sqrt{\Delta}$.   Using an independent transversal result of Haxell \cite{haxell2001note}, this was improved to $\omega \ge \frac34(\Delta + 1)$ in \cite{Rabern} and finally to the best possible $\omega > \frac23(\Delta + 1)$ in \cite{King}.  Using an independent transversal result of Alon~\cite{Alon} (see also~\cite{AS}, p.~70), we get $\omega\ge\frac{2e+1}{2e+2}(\Delta+1)$.  Since Alon's proof is based on the Local Lemma, we can use the efficient algorithms developed by Moser and Tardos \cite{moser2010constructive}.

\begin{lemma}
\label{Alonlemma}
If $G$ is a graph with $\omega\ge\frac{2e+1}{2e+2}(\Delta+1)$, then $G$ contains
an independent set $I$ such that $I$ intersects every maximum clique in $G$. 
\end{lemma}
\begin{proof}
Let $\S$ be the set of maximum cliques in $G$ and let $\S_i$ be the set of
vertices in one component $C_i$ of $X_{\S}$.  For each $i$,
Lemma~\ref{kostochka} gives $\card{\bigcap\S_i}\ge 2\omega-(\Delta+1)\ge
\frac{e}{e+1}(\Delta+1)$.

Let $k=\ceil{\frac{e}{e+1}(\Delta+1)}$.
For each component $C_i$, let $A_i$ be a set of $k$ vertices that lie in every
clique of $C_i$.  Use the Local Lemma (see \cite{AS}, p.~64--65) to choose the
desired independent set.  From each $A_i$, choose a vertex uniformly at random.
For each edge $uv$ with $u\in A_i$ and $v\in A_j$ (and $i\ne j$), let $E_{uv}$
be the bad event that both $u$ and $v$ are chosen for $I$; event $E_{uv}$
occurs with probability $p=1/(\card{A_i}\card{A_j})=k^{-2}$.  Each $E_{uv}$ is
independent of all other bad events except for those corresponding to edges
with an endpoint in $A_i$ or $A_j$.  Since each $u$ has at least $\omega - 1$ neighbors in $\S_i$ and $v$ has at least $\omega - 1$ neighbors in $\S_j$, the degree $d$ of $E_{uv}$ in the
dependency graph is at most $(\Delta + 1 - \omega) (|A_i| + |A_j|) - 1 \le \frac{2k}{2e+2}(\Delta + 1) - 1 = \frac{k}{e+1}(\Delta + 1) - 1$.  This gives $ep(d+1) \le 1$, so the desired independent set $I$ exists.
\end{proof}

\begin{cor}\label{BiggerDeltaAlgorithm}
If $G$ is a graph with $\Delta\ge 37$ and $\omega=\Delta-4$, then $G$ contains
an independent set $I$ such that $I$ intersects every maximum clique in $G$.
Furthermore, $I$ can be found in polynomial time.
\label{hittingsetcor}
\end{cor}
\begin{proof}
If $\Delta \ge 37$, then we have $\omega = \Delta - 4 \ge \frac{2e+1}{2e+2}(\Delta+1)$, 
so we can apply Lemma~\ref{Alonlemma}.  All that remains is to
show that we can implement its proof in polynomial time.  We can find the set
of all maximum cliques by considering each $(\Delta-4)$-element subset of the
closed neighborhood of each vertex.  We use a union-find algorithm to find the
components of the intersection graph of this set of maximum cliques.
Now consider a set $\S$ of maximum cliques such that the intersection graph
$X_{\S}$ is connected.  
We can slightly modify the union-find algorithm so that it also returns $\cap \S$.
To now find our hitting set, we apply the algorithm for the Local Lemma from Moser and Tardos \cite{moser2010constructive}.
\end{proof}

With a more complicated algorithm we can do better.  Specifically, instead of using Lemma \ref{hajnal} and Lemma \ref{kostochka}, we use Lemma \ref{K_tClassification} as in the proof of Lemma \ref{ourHittingLemma}.  Basically, we just need to do a preprocessing step where we find and remove all $d_1$-choosable induced subgraphs on at most $9$ vertices (we can color them after coloring the rest).  Once we have a graph with none of these $d_1$-choosable induced subgraphs, we know, as in the proof of Lemma \ref{ourHittingLemma}, that the components of $X_\S$ have at most two vertices.  So, we can replace our estimate $\card{\bigcap\S_i}\ge 2\omega-(\Delta+1)$ with $\card{\bigcap\S_i}\ge \omega - 1$.  This improves the needed condition in Lemma \ref{Alonlemma} to $\omega\ge\frac{2e}{2e+1}\Delta + 1$ and thus allows Corollary \ref{BiggerDeltaAlgorithm} to work for $\Delta \ge 33$.

Using a recent result of Kolipaka, Szegedy and Xu \cite{kolipaka2012sharper} we can do a bit better.  
The idea is that the local lemma can be strengthened when the dependency graph has nice structure.  In our case, the dependency graph is the line graph of a multigraph (the multigraph formed by contracting all the $A_i$ in $G\brackets{\bigcup_i A_i}$).  Because of this structure, we may apply the Clique Lov{\'a}sz Local Lemma from \cite{kolipaka2012sharper} to prove Lemma \ref{Alonlemma} with $\omega\ge\frac{4}{5}\Delta + 1$.  Since there is an efficient algorithm for the Clique Lov{\'a}sz Local Lemma as well, we get Corollary \ref{BiggerDeltaAlgorithm} for $\Delta\ge 26$.   So, we can prove the following conjecture for $\Delta \ge 26$.

\begin{conj}\label{AlgoConjecture}
For $\Delta \ge 13$, there is a polynomial time graph algorithm that finds either a $(\Delta-1)$-coloring or a clique on $\Delta-3$ vertices.
\end{conj}

\section*{Acknowledgments}
Thanks to the referees for numerous helpful suggestions, which markedly improved
the paper.


\begin{thebibliography}{99}
\begin{scriptsize}
\bibitem{Alon}
N. Alon, \emph{Probabilistic methods in coloring and decomposition problems}, 
Discrete Mathematics {\bf 127}, pp. 31--46, 1994.

\bibitem{AS}
N. Alon and J.H. Spencer, The Probabilistic Method,
Second Edition. (2000). New York: Wiley-Interscience.

\bibitem{baetz2001brooks}
B.~Baetz and D.~R.~Wood,
\emph{Brooks' vertex-colouring theorem in linear time}, 2001.
\url{http://arxiv.org/abs/1401.8023}

\bibitem{borodin1976decomposition}
O.V. Borodin, \emph{On decomposition of graphs into degenerate subgraphs},
Metody Diskretn. Analiz \textbf{28} (1976), 3--11 (in Russian).

\bibitem{BK}
O.V. Borodin and A.V. Kostochka, \emph{On an upper bound of a graph’s
chromatic number, depending on the graph’s degree and density}, J. of
Combin. Theory Ser. B {\bf 23} (1977), no. 2--3, pp.~247--250.

\bibitem{borodin2000variable}
O.V. Borodin, A.V. Kostochka, and B.~Toft, \emph{{Variable degeneracy:
extensions of Brooks' and Gallai's theorems}}, Discrete Mathematics
\textbf{214} (2000), no.~1-3, 101--112.

\bibitem{CatlinAnotherBound}
P.A. Catlin, \emph{{Another bound on the chromatic number of a graph}},
Discrete Mathematics \textbf{24} (1978), no.~1, 1--6.
  
\bibitem{CEK}
D. Christofides, K. Edwards, and A.D. King, \emph{A note on hitting maximum and
maximal cliques with a stable set},  J.\ of Graph Theory, to appear,
\url{http://arxiv.org/abs/1109.3092}

\bibitem{transitivenote}
D.W. Cranston and L. Rabern, \emph{A note on coloring vertex-transitive graphs}.

\bibitem{CR1}
D.W. Cranston and L. Rabern, \emph{Conjectures equivalent to the Borodin-Kostochka
Conjecture that are a priori weaker},  preprint,  \url{http://arxiv.org/abs/1203.5380}

\bibitem{big-cliques-arxiv1}
D.W. Cranston and L. Rabern, \emph{Graphs with $\chi=\Delta$ have big cliques},
preprint,  \url{http://arxiv.org/abs/1305.3526v2}

\bibitem{CR2}
D.W. Cranston, and L. Rabern, \emph{Coloring Claw-Free Graphs with $\Delta-1$ Colors}, 
SIAM J. Discrete Math. {\bf 27} (2013), no. 1, pp. 534--549.

\bibitem{Hajnal}
A.~Hajnal. \emph{A theorem on k-saturated graphs}, Canadian J. Math., {\bf 17},
pp. 720--724, 1965.

\bibitem{haxell2001note}
P.E.~Haxell.
\emph{A note on vertex list colouring},
Combinatorics, Probability and Computing, 10(04):345--347, 2001.

\bibitem{Kierstead}
H.A. Kierstead, \emph{On the choosability of complete multipartite graphs with
part size three}, Discrete Math. {\bf 211} (2000), no. 1--3, pp. 255--259.

\bibitem{kierstead2009ore}
H.A. Kierstead and A.V. Kostochka, \emph{{Ore-type versions of Brooks'
  theorem}}, J. of Combin. Theory Ser. B {\bf 99} (2009), no.~2, pp. 298--305.

\bibitem{King}
A.D.~King, \emph{Hitting all maximum cliques with a stable set using lopsided
independent transversals}, J. of Graph Theory {\bf 67} (2011), pp. 300-–305. 

\bibitem{KR}
A.D.~King and B.~Reed, \emph{A short proof that $\chi$ can be bounded $\epsilon$
away from $\Delta+1$ towards $\omega$}, preprint,
\url{http://arxiv.org/abs/1211.1410v1}

\bibitem{kolipaka2012sharper}
K.~Kolipaka, M.~Szegedy, and Y.~Xu.
\newblock A sharper local lemma with improved applications.
\newblock In {\em Approximation, Randomization, and Combinatorial Optimization.
  Algorithms and Techniques}, pages 603--614. Springer, 2012.

\bibitem{Kostochka}
A.V.~Kostochka, \emph{Degree, density, and chromatic number of graphs}, Metody
Diskret. Analiz. (In Russian), {\bf 35} (1980), pp. 45--70. 

\bibitem{KRS}
A.V. Kostochka, L. Rabern, M. Stiebitz, \emph{Graphs with chromatic number
close to maximum degree}, Discrete Math. {\bf 312} (2012), no. 6, pp. 1273--1281.

\bibitem{lovasz1966decomposition}
L.~Lov\'{a}sz, \emph{On decomposition of graphs}, Studia Sci. Math. Hungar.
  \textbf{1} (1966), 237--238.
  
\bibitem{lovasz1975three}
L.~Lov{\'a}sz,
\emph{Three short proofs in graph theory},
J. Combin. Theory Ser. B, {\bf 19} (1975), no. 3, pp. 269--271.
  
\bibitem{moser2010constructive}
R.~Moser and G.~Tardos.
\emph{A constructive proof of the general Lov{\'a}sz local lemma},
J. of the Association for Computing Machinery, {\bf 57} (2010), no. 2.
Art. 11, 15pp.  Also: \url{http://arxiv.org/abs/0903.0544}.

\bibitem{Moz1}
N.N. Mozhan, \emph{Chromatic number of graphs with a density that does not
exceed two-thirds of the maximal degree}, Metody Diskretn. Anal. {\bf 39}
(1983), pp.~52--65 (in Russian).

\bibitem{denseneighborhoods}
Landon Rabern, \emph{Coloring graphs with dense neighborhoods}, J. Graph Theory (2013).
  
\bibitem{Rabern}
L. Rabern, \emph{On hitting all maximum cliques with an independent set}, J.
Graph Theory {\bf 66} (2011), no. 1, pp.~32--37. 
 
\bibitem{Rabern2}
L. Rabern, \emph{$\Delta$-critical graphs with small high vertex cliques}, 
J. Combin. Theory Ser. B 
{\bf 102} (2012), no. 1, pp.~126--130.

\bibitem{Rabern3}
L. Rabern, \emph{Partitioning and coloring graphs with degree constraints}, 
Discrete Math.
{\bf 313} (2013), no. 9, pp.~1028--1034.

\bibitem{Reed}
B. Reed, \emph{A strengthening of Brooks' theorem}, J. Combin. Theory Ser. B
{\bf 76} (1999), no. 2, pp.~136--149.

\bibitem{RS}
B. Reed and B. Sudakov, \emph{List colouring when the chromatic number is close
to the order of the graph}, Combinatorica {\bf 25} (2004), no. 1, pp.~117--123.

\bibitem{IGT}
D.B. West, \textbf{Introduction to Graph Theory}, 
Prentice Hall, Inc., Upper Saddle River, NJ, 1996.

\end{scriptsize}
\end{thebibliography}
\end{document}